\documentclass[a4paper,11pt,openany]{amsart}
\usepackage[centertags]{amsmath}
\usepackage{amscd}
\usepackage{amsthm}
\usepackage{amssymb}
\usepackage{array}
\usepackage{enumerate}
\usepackage{hyperref}
\usepackage{cleveref}
\usepackage{multicol}
\usepackage{centernot}
\usepackage{mathtools}
\usepackage{stmaryrd}
\usepackage{tikzsymbols}
\usepackage{textcomp}
\usepackage{parskip}
\usepackage{mathrsfs}
\usepackage{tikz}
\usetikzlibrary{cd}
\usepackage{tikz-cd}
\usetikzlibrary{matrix}

\usepackage[utf8]{inputenc}
\usepackage[T1]{fontenc}

\usepackage[english]{babel}
\usepackage{color}
\usepackage{tikz}
\usepackage{indentfirst} 
\usepackage[T1]{fontenc}
\usepackage[sc]{mathpazo}
\usepackage[T1]{fontenc}
\usepackage[sc]{mathpazo}

\usepackage{geometry}
\usepackage{subfiles}
\usepackage{csquotes}
\usepackage{fancyhdr}

\newtheorem{defn}{Definition}[section]
\newtheorem{prop}[defn]{Proposition}
\newtheorem{thm}[defn]{Theorem}
\newtheorem{lem}[defn]{Lemma}
\newtheorem{cor}[defn]{Corollary}
\newtheorem{ex}[defn]{Example}
\newtheorem{rem}[defn]{Remark}
\newtheorem{conj}[defn]{Conjecture}
\newtheorem{ques}[defn]{Question}
\newtheorem{nota}[defn]{Notation}

\newenvironment{corollary}{\bigskip \begin{cor}}{\end{cor}}

\newcommand{\Z}{\mathbb{Z}}
\newcommand{\Q}{\mathbb{Q}}
\newcommand{\R}{\mathbb{R}}

\newcommand{\im}{\operatorname{Im}}
\renewcommand{\ker}{\operatorname{Ker}}

\newcommand{\Homeo}{\operatorname{Homeo}}

\newcommand{\Out}{\operatorname{Out}}
\newcommand{\rank}{\operatorname{rank}}
\newcommand{\Aut}{\operatorname{Aut}}
\newcommand{\Inn}{\operatorname{Inn}}
\newcommand{\SO}{\operatorname{SO}}

\newcommand{\Gl}{\operatorname{GL}}
\newcommand{\Sl}{\operatorname{SL}}

\newcommand{\Zc}{\mathcal{Z}}

\newcommand{\D}{\operatorname{disc-sym}}
\newcommand{\A}{\operatorname{tor-sym}}
\newcommand{\acts}{\curvearrowright}
\newcommand{\Stab}{\operatorname{Stab}}
\newcommand{\Fitt}{\operatorname{Fitt}}

\newcommand{\Aff}{\operatorname{Aff}}
\newcommand{\vol}{\operatorname{vol}}

\begin{document}
	
	\title{Iterated finite group actions on closed connected aspherical manifolds}
	\author{Jordi Daura Serrano}
	\address{Jordi Daura Serrano, Department de Màtematiques i Informàtica, Universitat de Barcelona (UB), Gran Via de les Corts Catalanes 585, 08007 Barcelona (Spain)}
	\email{jordi.daura@ub.edu}
	
	\maketitle
	\pagenumbering{arabic}
	\begin{abstract}
	In this paper we use free iterated actions and the iterated discrete degree of symmetry introduced in \cite{daura2025IGG1} to obtain rigidity results on aspherical manifolds. We also introduce the concept of the length of an iterated action and we study it for nilmanifolds, solvmanifold and locally symmetric spaces.
	\end{abstract}
	\date{}

\section{Introduction}

Let $M$ be a closed connected manifold. Can we determine $M$ up to homeomorphism if we know the collection of finite group which acts effectively on $M$? The answer to this question is clearly negative since there exist many closed connected asymmetric manifolds (manifolds which do not admit any effective group action of a non-trivial finite group) which are not homeomorphic. Indeed, it is conjectured in \cite{puppe2007manifolds} and \cite{schultz1981group} that "most" manifolds are asymmetric. The question becomes interesting if we assume that $M$ admits effective group actions of large finite groups in some sense. A way to make this question precise is to use the discrete degree of symmetry introduced in \cite{mundet2021topological}. 

\begin{defn}\label{def:discsym intro}
	Given a manifold $M$ let 
	$$\mu(M)=\{r\in\mathbb{N}:\text{ $M$ admits an effective action of $(\Z/a)^r$ for arbitrarily large $a$}\}.$$
	
	More explicitly, $r\in \mu(M)$ if there exists an increasing sequence of natural number $\{a_i\}$ and effective group actions of $(\Z/a_i)^r$ on $M$ for each $i$.
	
	The discrete degree of symmetry of a manifold $M$ is
	$$\D(M)=\max(\{0\}\cup\mu(M)).$$
\end{defn}

The discrete degree of symmetry was introduced in \cite{mundet2021topological}, and it can be considered an analogue for finite group actions of the toral degree of symmetry, which is defined as $$ \A(M)=\max(\{0\}\cup\{r\in\mathbb{N}:T^r \text{ acts effectively on $M$}\}),$$ and which has been widely studied (see \cite[Chapter VII. \S2]{hsiang2012cohomology}, \cite[\S 11.7, \S11.8]{lee2010seifert} and the survey \cite{grove2002geometry}). Recall that if $M$ is a closed connected $n$-dimensional manifold, then $\A(M)\leq n$ and the equality holds if and only if $M$ is homeomorphic to $T^n$ (see \cite[12.2]{mundet2021topological}). A natural question to ask is whether this result is still true if we replace $\A(M)$ by $\D(M)$:

\begin{ques}\cite[Question 3.4, Question 3.5]{riera2023actions}
Let $M$ be a closed connected $n$-dimensional manifold. Is $\D(M)\leq n$? If $\D(M)=n$, is $M$ homeomorphic to $T^n$?
\end{ques}

These questions have been affirmatively answered for some classes of manifolds:

\begin{thm}\label{discsym rigidity}
Let $M$ be a closed connected $n$-dimensional manifold. Then:
\begin{itemize}
	\item[(1)] \cite[Theorem 1.3]{mundet2021topological} If $M$ is oriented and there exists a non-zero degree map $M\longrightarrow T^n$ then $\D(M)\leq n$ and if the equality holds and $\pi_1(M)$ is virtually solvable then $M\cong T^n$.
	\item[(2)] \cite[Theorem 1.6, Theorem 1.9]{daura2024actions} If $M$ is aspherical and there exists a constant $C$ such that every finite subgroup $G$ of $\Out(\pi_1(M))$ satisfies $|G|\leq C$, then $\D(M)\leq n$ and the equality holds if and only if $M\cong T^n$.
\end{itemize} 
\end{thm}

\begin{rem}
	Following the notation introduced in \cite{daura2024actions}, we will say that a group $\mathcal{G}$ is Minkowski if there exists a constant $C$ such that every finite subgroup $G\leq \mathcal{G}$ satisfies $|G|\leq C$. This name is motivated by a classical result of Hermann Minkowski which states that $\Gl(n,\Z)$ is Minkowski. The Minkowski property was studied under the name of bounded finite subgroups property in \cite{popov2011makar,prokhorov2014jordan,popov2018jordan,golota2023finite,bandman2024jordan}.
\end{rem}

The discrete degree of symmetry is useful to detect when a closed connected manifold is homeomorphic to $T^n$. However, it falls short to obtain rigidity results for other classes of manifolds. A class of manifolds which generalize tori are nilmanifolds, which are coset spaces $N/\Gamma$ where $N$ is a simply connected nilpotent Lie group and $\Gamma$ is a lattice of $N$. An example of nilmanifold in dimension $3$ are Heisenberg manifolds, which can be obtained as the total space of a non-trivial principal $S^1$-bundle over $T^2$. It is straightforward to prove that their discrete degree of symmetry is equal to $1$. However, there exist many other $3$-dimensional manifolds with discrete degree of symmetry equal to 1 which are not homeomorphic to a Heisenberg manifold (see \cref{main theorem10 intro} below and \cite{daura2024actions}).

Motivated by the fact that nilmanifolds are precisely the total space of iterated principal $S^1$-bundles (see \cite{belegradek2020iterated}), we introduced free iterated actions of finite groups and the iterated discrete degree of symmetry in \cite{daura2025IGG1} to generalize \cref{discsym rigidity}.(2). Iterated group actions were previously studied with other purposes by other authors, see \cite{baker2001towers,figueroa2012half,van2018towers,van2021structure,baues2023isometry,qin2021self}.

Let us briefly recall some of the concepts introduced in \cite{daura2025IGG1}.
 
\begin{defn}
	Let $\mathcal{G}=\{G_i\}_{i=1,\dots,n}$ be a collection of groups and let $X$ be a topological space. An iterated action of $\mathcal{G}$ on $X$ (denoted by $ \mathcal{G}\acts X$) is:
\begin{itemize}
	\item[1.] A sequence of surjections of topological spaces
	\[\begin{tikzcd}
		X=X_0 \ar{r}{p_1} & X_1 \ar{r}{p_2} & X_2 \ar{r}{p_3} & \cdots \ar{r}{p_n} & X_n,
	\end{tikzcd}\]
	\item[2.] and a collection of group actions $\{\Phi_i:G_i\longrightarrow \Homeo(X_{i-1})\}_{i=1,\dots,n}$,
\end{itemize}
such that the maps $p_i:X_{i-1}\longrightarrow X_i$ are the orbit maps of the action of $G_i$ on $X_{i-1}$. If all actions are free we say that the iterated action is free.

Let $\mathcal{G}=\{G_i\}_{i=1,...,n}$ and $\mathcal{G}'=\{G'_i\}_{i=1,...,n'}$ be  two collections of finite groups which act freely on $M$. Let $p=p_{n}\circ\cdots \circ p_1$ and $p'=p'_{n'}\circ\cdots\circ p'_1$. We say that the iterated actions $\mathcal{G}\acts M$ and $\mathcal{G}'\acts M$ are equivalent (and we denote it by $\mathcal{G}\acts M\sim\mathcal{G}'\acts M$) if there exist homeomorphisms $f:M_n\longrightarrow M_{n'}$ and $\overline{f}:M\longrightarrow M$ such that $p'\circ \overline{f}=f\circ p$ (that is, $p$ and $p'$ are isomorphic coverings).

The equivalence class will be denoted by $[\mathcal{G}\acts M]$. If there exists a finite group $G$ such that $\mathcal{G}\acts M$ is equivalent to $\{G\}\acts M$ then we say that $\mathcal{G}\acts M$ is simplifiable.

Given a free iterated action $\mathcal{A}\acts M$ of finite abelian groups, the rank of the iterated action is $$\rank_{ab}(\mathcal{A}\acts M)=\min\{\sum_{i=1}^{n}\rank A'_i:\{A_1',\cdots,A_n'\}\acts M\in [\mathcal{A}\acts M]\text{, $A_i'$ abelian for all $i$}\}.$$	
We define $\mu_2(M)$ as the set of all pairs $(f,b)\in \mathbb{N}^2$ which satisfy:
\begin{itemize}
	\item[1.] There exist an increasing sequence of prime numbers $\{p_i\}$, a sequence of natural numbers $\{a_i\}$ and a collection of free iterated actions $\{(\Z/p_i^{a_i})^f,(\Z/p_i)^b\}\acts M$ for each $i\in\mathbb{N}$.
	\item[2.] $\rank_{ab}(\{(\Z/p_i^{a_i})^f,(\Z/p_i)^b\}\acts M)=f+b$ for each $i\in\mathbb{N}$.
\end{itemize}

Consider the lexicographic order in $\mathbb{N}^2$, that is $(a,b)\geq (c,d)$ if $a>c$, or $a=c$ and $b\geq d$. Define the iterated discrete degree of symmetry of $M$ as
$$\D_2(M)=\max\{(0,0)\cup\mu_2(M)\}.$$
\end{defn}

The iterated discrete degree of symmetry of a manifold is a refinement of the degree of symmetry of a manifold. The next theorem, proved in \cite{daura2025IGG1}, is a generalization of \Cref{discsym rigidity}.(1).

\begin{thm}\cite{daura2025IGG1}
	Let $M$ be a closed connected orientable $n$-dimensional manifold admitting a non-zero degree map $M\longrightarrow N/\Gamma$ to a $2$-step nilmanifold satisfying that $\rank\Zc\Gamma=f$. Then $\D_2(M)\leq (f,n-f)$. Moreover, if $\D_2(M)=(f,n-f)$ and $\pi_1(M)$ is virtually solvable then $M\cong N/\Gamma$.
\end{thm}

Thus, the iterated discrete degree of symmetry is a suitable tool to study rigidity problems on manifolds related with 2-step nilmanifolds. The main goal of this paper is to generalize \cref{discsym rigidity}.(2). In particular, we prove:

\begin{thm}\label{main theorem9 intro}
	Let $M$ be a closed connected $n$-dimensional aspherical manifold such that $\Zc\pi_1(M)$ and $\Zc\Inn\pi_1(M)$ are finitely generated, and $\Aut(\Inn\pi_1(M))$ and $\Out(\Inn\pi_1(M))$ are Minkowski. Assume that $\D_2(M)=(f,b)$ with $f+b=n$. Then $M$ is homeomorphic to a 2-step nilmanifold $N/\Gamma$ satisfying that $\rank\Zc\Gamma=f$.
\end{thm}

It is an interesting question whether all closed connected aspherical manifolds satisfy the hypothesis of \cref{main theorem9 intro}.

No closed aspherical manifold with $\Zc\pi_1(M)$ not finitely generated is known (see \cite[Remark 3.1.19.]{lee2010seifert}). In \cite{belk2022embedding}, Belk, Hyde and Matucci constructed a finitely presented group $G$ such that $\Zc G\cong (\Q,+)$, which is not finitely generated (see also \cite[Theorem II]{houcine07}). We do not know if $G$ could be the fundamental group of a closed connected aspherical manifold (we do not even know if $G$ is torsion-free or the cohomology of $G$ is finite and satisfies Poincaré duality).

Regarding the conditions on the automorphism and outer automorphism groups, we prove:

\begin{prop}\label{autoutinn minkowski}
Let $\Gamma$ be a lattice of a connected Lie group $G$, then $\Aut(\Inn\Gamma)$ and $\Out(\Inn\Gamma)$ are Minkowski.
\end{prop} 

\Cref{autoutinn minkowski} can be deduced using the same arguments as in the proof of \cite[Thereom 1.9]{daura2024actions}. Thus, we can use \cref{main theorem9 intro} on locally homogeneous spaces $ H\setminus G/\Gamma$, where $H$ is a connected Lie group, $H$ is a maximal compact subgroup of $G$ and $\Gamma$ is a cocompact torsion-free lattice of $G$. 

We also compute the iterated discrete degree of symmetry for closed connected aspherical 3-manifolds. In the following theorem, $K$ denotes the Klein bottle and $SK$ denotes the total space of the non-trivial principal $S^1$-bundle over $K$.

\begin{thm}\label{main theorem10 intro}
	Let $M$ be a $3$-dimensional closed connected aspherical manifold. Then:
	\begin{itemize}
		\item[1.] $\D_2(M)=(3,0)$ if $M\cong T^3$.
		\item[2.] $\D_2(M)=(2,0)$ if $M\cong K\times S^1$ or $M\cong SK$.
		\item[3.] $\D_2(M)=(1,2)$ if $M$ is a Heisenberg manifold.
		\item[4.] $\D_2(M)=(1,0)$ if $\Zc\pi_1(M)\cong \Z$ and $\Inn \pi_1(M)$ is centreless.
		\item[5.] $\D_2(M)=(0,0)$ if $M$ does not belong to any of the above 4 cases.
	\end{itemize}
\end{thm}

In \cref{main theorem9 intro} and \cref{main theorem10 intro}, we only consider free iterated actions of two finite groups. We can study, instead, free iterated actions on manifolds of many finite groups. To do so, we define the following invariant:

\begin{defn}\label{lenght def}
Given a collection of finite groups $\mathcal{G}=\{G_1,\dots,G_n\}$, we define $l(\mathcal{G})=n$. Given a free iterated action $\mathcal{G}\acts M$, the length of the iterated action is $$l(\mathcal{G}\acts M)=\min\{l(\mathcal{G}'):\mathcal{G}'\acts M\in [\mathcal{G}\acts M]\}.$$
The iterated length of a manifold $M$ is $$l(M)=\max\{l(\mathcal{G}\acts M): \text{ free iterated action }\mathcal{G}\acts M\}.$$
\end{defn}

With this notation, a free iterated action $\mathcal{G}\acts X$ is simplifiable if and only if $l(\mathcal{G}\acts X)=1$ and all free iterated actions on $X$ are simplifiable if and only if $l(X)=1$. The next result gives a bound for the iterated length of a nilmanifold.

\begin{thm}\label{main theorem8 intro}
	Let $N/\Gamma$ be a $c$-step nilmanifold. There exists a constant $C$ such that any free iterated action $\mathcal{G}\acts N/\Gamma$ is equivalent to a free iterated action $\mathcal{G}'\acts N/\Gamma$ where $\mathcal{G}'=\{A_1,...,A_c,G'\}$, $A_i$ are finite abelian groups such that $\rank(A_i)\leq b_i$ and $|G'|\leq C$. In particular, $l(N/\Gamma)\leq c+1$.
\end{thm}

Given a closed connected manifold $M$, it is an interesting question to study when $l(M)$ is bounded. If no bound exists then we will write $l(M)=\infty$. We have seen that $l(M)$ is bounded if $M$ is a nilmanifold. Another case where the iterated length is bounded is the following:

\begin{prop}\label{main theorem5 intro}
Given a locally symmetric space $H\setminus G/\Gamma$, there exists $C$ depending on $\Gamma$ such that $l(M)\leq C$.
\end{prop}

However, not all closed connected aspherical manifolds have bounded iterated length, as the next result shows. 

\begin{thm}\label{main theorem6 intro}
	\begin{itemize}
		\item[1.] There exists a closed solvmanifold $M$ such that $l(M)=\infty$.
		\item[2.] There exist a closed connected aspherical locally homogeneous space $H\setminus G/\Gamma$ such that the solvable radical of $G$ is abelian and $l(H\setminus G/\Gamma)=\infty$.
	\end{itemize}
\end{thm}

This paper is divided as follows. In the first section we recall some preliminary group theory results which will be widely used to prove the main results of this paper. The second section is devoted to prove \cref{main theorem9 intro}. In section three we prove \cref{main theorem10 intro}. Finally, in the forth section we study the length of free iterated actions on aspherical manifolds, proving \cref{main theorem5 intro}, \cref{main theorem6 intro} and \cref{main theorem8 intro}.

\textbf{Acknowledgements:} I would like to thank Ignasi Mundet i Riera for all the guidance given during this project and for the careful revision of the first draft of this paper. Many thanks to Leopold Zoller for helpful
comments and corrections. 

This work was partially supported by the grant PID2019-104047GB-I00 from the Spanish Ministry of Science and Innovation and the Departament de Recerca i Universitats de la Generalitat de Catalunya (2021 SGR 00697).

\section{Preliminaries}

In this section we recall some results which will be widely used across the article.

\subsection{Outer automorphism group}\label{sec:preliminares}

The aim of this section is to briefly explain the constructions in \cite{MalfaitWim2002Toag} used to compute the outer automorphism group of a group extension. 

Let
\[\begin{tikzcd}
	1\ar{r}{}&K\ar{r}{}&G\ar{r}{p}&Q\ar{r}{}&1
\end{tikzcd}\]
be a short exact sequence of groups. The extension is determined by the morphism $\psi:Q\longrightarrow\Out(K)$ and a 2-cocycle $c\in H^2_\psi(Q,\Zc K)$. Let $\Aut(G,K)=\{f\in\Aut(G):f(K)=K\}$ and $\Out(G,K)=\Aut(G,K)/\Inn G$. Define the group morphism $\Theta:\Aut(G,K)\longrightarrow \Aut(K)\times \Aut(Q)$ such that $f\mapsto(f_{|K},\overline{f})$, where $\overline{f}:Q\longrightarrow Q$ is the map induced by $f$ on $Q$. Finally, recall that if $H$ is a subgroup of $G$, the centralizer $C_G(H)$ is $\{g\in G:c_h(g)=g\text{ for all }h\in H\}$.

\begin{thm}\cite[Theorem 3.3, Theorem 3.6]{MalfaitWim2002Toag}
	With the above notation, we have
	\begin{itemize}
		\item[1.] $\im(\Theta)$ are the elements of $\Aut(K)\times\Aut(Q)$ which fix the 2-cocycle $c$.
		\item[2.] There exists an isomorphism $\xi:\ker(\Theta)\longrightarrow Z^1_\psi(Q,\Zc K)$.
	\end{itemize}
\end{thm}

\begin{defn}\cite[Definition 3.7]{MalfaitWim2002Toag}
	We define
	$$\overline{B}_\psi^1(Q,\Zc K)=\{\xi(c_g):g\in C_G(K),p(g)\in\Zc Q\}. $$
\end{defn} 

\begin{prop}\cite[Proposition 3.8]{MalfaitWim2002Toag}
	We have $$B_\psi^1(Q,\Zc K)\leq\overline{B}_\psi^1(Q,\Zc K)\leq Z^1_\psi(Q,\Zc K).$$ In consequence, there exists a surjective morphism $$H^1_\psi(Q,\Zc K)\longrightarrow \overline{H}_\psi^1(Q,\Zc K)=Z^1_\psi(Q,\Zc K)/\overline{B}_\psi^1(Q,\Zc K).$$ 
	
	We have an isomorphism $\overline{B}_\psi^1(Q,\Zc K)\cong (p^{-1}(\Zc Q)\cap C_GK)/\Zc G$.
\end{prop}

We can think $\overline{H}_\psi^1(Q,\Zc K)$ as the subgroup of $\Out(G,K)$ whose automorphism classes induce inner automorphisms on $K$ and $Q$. This interpretation is made precise in the next theorem.

\begin{thm}\label{out ses}\cite[Theorem 4.8]{MalfaitWim2002Toag}
	There exist short exact sequences
	\[\begin{tikzcd}
		1\ar{r}{}&\mathcal{K}\ar{r}{}&\Out(G,K)\ar{r}{}&L_1\ar{r}{}&1,
	\end{tikzcd}\]
	and
	\[\begin{tikzcd}
		1\ar{r}{}&\overline{H}_\psi^1(Q,\Zc K)\ar{r}{}&\mathcal{K}\ar{r}{}&L_2\ar{r}{}&1
	\end{tikzcd}\]
	where $$L_1=\{\overline{f}\in\Aut(Q):f\in\Aut(G,K)\}/\Inn(Q)\leq \Out(Q)$$ and $$L_2\cong (\Stab_{\Aut(K)}c/\Inn(K))/\Zc\psi(Q)\leq C_{\Out(K)}\psi(Q)/\psi(\Zc Q).$$
\end{thm} 

There is also a version for the automorphism group, which was also investigated by C. Wells in \cite{wells1971automorphisms}.

\begin{thm}\label{aut ses}\cite[Theorem 4.8]{MalfaitWim2002Toag}
	There exist short exact sequences
	\[\begin{tikzcd}
		1\ar{r}{}&\mathcal{K}'\ar{r}{}&\Aut(G,K)\ar{r}{}&L'_1\ar{r}{}&1,
	\end{tikzcd}\]
	and
	\[\begin{tikzcd}
		1\ar{r}{}&Z_\psi^1(Q,\Zc K)\ar{r}{}&\mathcal{K}'\ar{r}{}&L'_2\ar{r}{}&1
	\end{tikzcd}\]
	where $$L'_1=\{\overline{f}\in\Aut(Q):f\in\Aut(G,K)\}\leq \Aut(Q)$$ and $$L'_2=\Stab_{\Aut(K)}c\leq \Aut(K).$$
\end{thm} 

\begin{rem}
	If $K$ is a characteristic subgroup of $G$ then $\Aut(G,K)=\Aut(G)$ and $\Out(G,K)=\Out(G)$.
\end{rem}

\begin{lem}\label{Minkowski ses}\cite[Lemma 2.4]{daura2024actions}
	Let $\begin{tikzcd}		1\ar{r}{}& K\ar{r}{}&H\ar{r}{p}&Q\ar{r}{}&1
	\end{tikzcd}$ be a short exact sequence of groups. If $K$ and $Q$ are Minkowski, then $H$ is Minkowski. If $K$ is finite and $H$ is Minkowski then $Q$ is Minkowski.
\end{lem}

\begin{lem}\label{outer automorphism and finite index subgroups}\cite[Corollary 2.4]{daura2024actions}
	Let $\Gamma$ and $\Gamma'$ be finitely generated groups with finitely generated center such that $\Gamma'\trianglelefteq\Gamma$ and $\Gamma/\Gamma'=F$ is a finite group. If $\Out(\Gamma')$ is a Minkowski, then $\Out(\Gamma)$ is Minkowski.
\end{lem}

\subsection{Infranilmanifolds and almost crystallographic groups}

Let $N$ be a simply connected nilpotent Lie group and $C$ be a maximal compact subgroup of $\Aut(N)$. A discrete cocompact subgroup $\Gamma$ of $N\rtimes C$ is called an almost-crystallographic group (abbreviated as AC-group). Moreover, $\Gamma$ is said to be almost-Bieberbach if it is torsion-free. If $\Lambda$ be a finitely generated torsion-free nilpotent group, an extension $1\longrightarrow \Lambda\longrightarrow \Gamma\longrightarrow G\longrightarrow 1$ in which $G$ is finite is said to be essential if $\Lambda$ is a maximal nilpotent subgroup of $\Gamma$. 
Note that crystallographic groups are an example of AC-groups. Bieberbach theorems for crystallographic groups have been generalized to almost-crystallographic groups.

\begin{thm}\label{generalized Bieberbach theorems}\cite[Chapter 2]{dekimpe2006almost}
	Let $N$ be a simply connected nilpotent Lie group and $C$ a maximal compact subgroup of $\Aut(N)$. 
	\begin{itemize}
		\item[1.] (Generalized 1st Bieberbach theorem) Let $\Gamma\leq N\rtimes C$ be an AC-group of $N$. The subgroup $N\cap \Gamma=\Lambda$ is a lattice of $N$, $\Lambda$ is the unique normal maximal nilpotent subgroup of $\Gamma$ and $\Gamma/\Lambda$ is a finite group.
		\item[2.] (Generalized 2nd Bieberbach theorem) Let $\Gamma$ and $\Gamma'$ be AC-groups of $N$. Assume that there exists an isomorphism $f:\Gamma\longrightarrow \Gamma'$. Then $f$ can be realized as a conjugation by an element of $\Aff(N)=N\rtimes\Aut(N)$.
		\item[3.] (Generalized 3rd Bieberbach theorem) Let $\Lambda$ be a lattice of $N$. There are only finitely many essential extensions $1\longrightarrow \Lambda\longrightarrow \Gamma\longrightarrow G\longrightarrow 1$.
	\end{itemize}
\end{thm}

If $\Gamma$ is an almost-Bieberbach group of $N$ then $\Gamma$ acts freely on $N$ and the orbit space $N/\Gamma$ is a closed connected aspherical manifold which is finitely covered by the nilmanifold $N/\Lambda$. These manifolds are known as infranilmanifolds, and just like flat manifolds, they also have a geometric characterization, see \cite{gromov1978almost,ruh1982almost}. The algebraic characterization of crystallographic groups (see \cite[Theorem 2.1.4]{dekimpe2016users}) can also be generalized to AC-groups, but we need some further definitions to state it.

\begin{defn}\cite[Definition 2.3.3]{dekimpe2006almost}
	Let $\Gamma$ be a virtually polycyclic group. The Fitting subgroup $\Fitt(\Gamma)$ is the unique maximal normal nilpotent subgroup of $\Gamma$. The Fitting subgroup $\Fitt(\Gamma)$ is unique.
\end{defn}

\begin{thm}\label{AC-groups charactherization}\cite[Theorem 3.4.6]{dekimpe2006almost}
	Let $\Gamma$ be a virtually polycyclic group. The following are equivalent:
	\begin{itemize}
		\item[1.] $\Gamma$ is an AC-group.
		\item[2.] $\Fitt(\Gamma)$ is torsion-free, maximal nilpotent and of finite index in $\Gamma$.
		\item[3.] $\Gamma$ contains a torsion-free nilpotent subgroup $\Lambda$ of finite index such that $C_\Gamma(\Lambda)$ is torsion-free.
		\item[4.] $\Gamma$ contains a nilpotent subgroup of finite index and $\Gamma$ does not contain any non-trivial finite normal subgroup.
	\end{itemize}
\end{thm}
\begin{corollary}
	If $\Gamma$ is a torsion-free virtually polycyclic group then $\Gamma$ is an AC-group if and only if it contains a nilpotent subgroup of finite index.
\end{corollary}

To announce the next propositions we need some facts on the rational Mal'cev completion of a torsion-free finitely generated nilpotent group. We refer to \cite{dekimpe2006almost,dekimpe2016users} for its definition and main properties. Let $\Lambda$ be a torsion-free finitely generated nilpotent group and let $\Lambda_\Q$ denote its rational Mal'cev completion. Given $f\in \Aut(\Gamma)$, we can uniquely extend $f$ to an automorphism $f_\Q:\Lambda_\Q\longrightarrow \Lambda_\Q$ of the rational Mal'cev completion such that $f_{\Q|\Lambda}=f$ (see \cite[Proposition 2.7]{dekimpe2016users}). Thus, we can define an injective group morphism $\Aut(\Lambda)\longrightarrow\Aut(\Lambda_\Q)$ which induces a morphism $\Inn\Lambda\longrightarrow \Inn\Lambda_\Q$. Consequently, we have a group morphism $\Out(\Lambda)\longrightarrow \Out(\Lambda_\Q)$. Given a group morphism $\phi:G\longrightarrow \Out(\Lambda)$, we denote by $\phi_\Q:G\longrightarrow \Out(\Lambda_\Q)$ the composition of $\phi$ and $\Out(\Lambda)\longrightarrow \Out(\Lambda_\Q)$. The proof of \cref{AC-groups charactherization} uses the following proposition, which will be also used in the proof of \cref{main theorem10 intro}.

\begin{prop}\label{infranilmanifolds malcev completion injective}
	Let $\Lambda$ be a finitely generated torsion-free nilpotent group and $1\longrightarrow \Lambda\longrightarrow \Gamma\longrightarrow G\longrightarrow 1$ a group extension where $G$ is finite. Then $\Lambda$ is maximal nilpotent in $\Gamma$ if and only if the induced map $\phi_\Q:G\longrightarrow \Out(\Lambda_\Q)$ is injective.
\end{prop}

This proposition is proved in \cite[Lemma 3.1.1]{dekimpe2006almost} for the real Mal'cev completion, and the same argument can be used to prove \cref{infranilmanifolds malcev completion injective} with the rational Mal'cev completion. It can be also reformulated in the following way:

\begin{prop}\label{infranilmanifolds malcev completion trivial}
	Let $\Lambda$ be a finitely generated torsion-free nilpotent group and $1\longrightarrow \Lambda\longrightarrow \Gamma\longrightarrow G\longrightarrow 1$ a group extension where $G$ is finite and $\Gamma$ is torsion-free. Then $\Gamma$ is nilpotent if and only if the induced map $\phi_\Q:G\longrightarrow \Out(\Lambda_\Q)$ is trivial.
\end{prop}

Recall that $\Lambda_\Q$ has associated a finite dimensional rational Lie algebra $\mathcal{L}(\Lambda)_\Q$ (for the details, see \cite[Theorem 2.11, Theorem 2.12]{dekimpe2016users}).

\begin{lem}\cite[Corollary 2.13]{dekimpe2016users}\label{automorphisms malcev completion}
	Let $\Lambda$ be a finitely generated torsion-free nilpotent group. Then $\Aut(\Gamma_\mathbb{Q})\cong\Aut({\mathcal{L}(\Gamma)}_\mathbb{Q})$.
\end{lem}

\subsection{Groups action on aspherical manifolds}

Let $M$ be a closed connected manifold and let $x_0\in M$. Assume that we have a finite group $G$ acting on $M$. Then for each $g\in G$ we have an isomorphism ${g_*}:\pi_1(M,x_0)\longrightarrow \pi_1(M,gx_0)$. If $x_0$ is fixed by the action of $G$ on $M$, then we can define a group morphism $G\longrightarrow \Aut(\pi_1(M,x_0))$. If $x_0$ is not fixed by the action of $G$ on $M$ then above group morphism is not well defined. However, since $\pi_1(M,x_0)$ and $\pi_1(M,gx_0)$ are isomorphic and the isomorphism is given by a conjugation, we have a well-defined group morphism $\psi:G\longrightarrow \Out(\pi_1(M,x_0))=\Aut(\pi_1(M,x_0))/\Inn(\pi_1(M,x_0))$ given by $\psi(g)=[{g_*}:\pi_1(M,x_0)\longrightarrow \pi_1(M,gx_0)]$. We will say that the action is inner if $\psi:G\longrightarrow \Out(\pi_1(M,x_0))$ is trivial. We will omit the base point whenever it is not necessary for the discussion.

We can always lift the action of $G$ on $M$ to an action of a group $\tilde{G}$ on the universal cover $\tilde{M}$, where the group $\tilde{G}$ fits into the short exact sequence
\[\begin{tikzcd}
	1\ar{r}{}& \pi_1(M)\ar{r}{} & \tilde{G} \ar{r}{}& G\ar{r}{} & 1.
\end{tikzcd}\]
The abstract kernel of the group extension coincides with the group morphism $\psi$.

\begin{lem}\cite[Lemma 3.1.14]{lee2010seifert}\label{aspherical manifolds big diagram}
	Let $G$ be a finite group acting effectively on a closed manifold $M$. Then, there is a commutative diagram with exact rows and columns
	\[\begin{tikzcd}
		& 1\ar{d}{} & 1 \ar{d}{}& 1 \ar{d}{} & \\
		1\ar{r}{}& \Zc\pi_1(M)\ar{r}{}\ar{d}{} & C_{\tilde{G}}(\pi_1(M)) \ar{r}{}\ar{d}{}& \ker\psi\ar{r}{}\ar{d}{} & 1\\
		1\ar{r}{}& \pi_1(M)\ar{r}{}\ar{d}{} & \tilde{G} \ar{r}{}\ar{d}{}& G\ar{r}{}\ar{d}{\psi} & 1\\
		1\ar{r}{}& \Inn(\pi_1(M))\ar{r}{}\ar{d}{} & \Aut(\pi_1(M)) \ar{r}{}& \Out(\pi_1(M))\ar{r}{} & 1\\
		&1 &&&
	\end{tikzcd}\]
	where $C_{\tilde{G}}(\pi_1(M))$ is the centralizer of $\pi_1(M)$ in $\tilde{G}$.
\end{lem}

\Cref{aspherical manifolds big diagram} can be used to prove the following result.

\begin{thm}\cite[Theorem 3.1.16]{lee2010seifert}\label{finite group actions aspherical manifolds}
	Let $G$ be a finite group acting effectively on a closed connected $n$-dimensional aspherical manifold $M$, then:
	\begin{itemize}
		\item[1.] $C_{\tilde{G}}(\pi_1(M))$ is torsion free.
		\item[2.] $\ker \psi$ is abelian. Moreover, if $\Zc\pi_1(M)$ is finitely generate of rank $k$, then $\ker\psi$ is a subgroup of the torus $T^k$.
	\end{itemize}
\end{thm}

\section{Free iterated actions on closed aspherical manifolds}\label{sec: iterated actions on aspherical manifolds}

The aim of this section is to prove \cref{main theorem9 intro}. To ease notation, we introduce the following definition:

\begin{defn}\label{2-step Minkowski}
	If $\Gamma$ is a finitely generated group satisfying that $\Zc\Gamma$ and $\Zc\Inn\Gamma$ are finitely generated and that $\Out(\Inn\Gamma)$ and $\Aut(\Inn\Gamma)$ are Minkowski then we say that $\Gamma$ has the 2-step Minkowski property.
\end{defn}

\begin{rem}\label{2-step Minkowski 1-cocycles}
	Assume that $\Gamma$ is 2-step Minkowski. Then $\Zc\Gamma$ and $\Inn\Gamma$ are finitely generated and therefore the group of closed 1-cocycles $Z^1(\Inn\Gamma,\Zc\Gamma)$ is a finitely generated abelian group, and hence $Z^1(\Inn\Gamma,\Zc\Gamma)$ is Minkowski.
\end{rem}

\begin{lem}
	Assume that $\Gamma$ is torsion-free and 2-step Minkowski. Then $\Out(\Gamma)$ and $\Aut(\Gamma)$ are Minkowski.
\end{lem}

\begin{proof}
	We consider the central short exact sequence $1\longrightarrow \Zc\Gamma\longrightarrow \Gamma\longrightarrow \Inn\Gamma\longrightarrow 1$. We know that $\Out(\Zc\Gamma)=\Aut(\Zc\Gamma)=\Gl(n,\Z)$ for some $n$, hence $\Out(\Zc\Gamma)$ is Minkowski. Moreover, $H^1(\Inn\Gamma,\Zc\Gamma)$ is Minkowski because $\Inn\Gamma$ and $\Zc\Gamma$ are finitely generated, and $\Out(\Inn\Gamma)$ is Minkowski by hypothesis. By \cref{out ses} and \cref{Minkowski ses}, we can conclude that $\Out(\Gamma)$ is Minkowski.
	
	The group $Z^1(\Inn\Gamma,\Zc\Gamma)$ is Minkowski by \cref{2-step Minkowski 1-cocycles} and $\Aut(\Inn\Gamma)$ is Minkowski by hypothesis. Therefore $\Aut(\Gamma)$ is Minkowski by \cref{aut ses} and \cref{Minkowski ses}.
\end{proof}

\begin{rem}
	Note that if $\Gamma$ is a centreless finitely generated group, then $\Inn\Gamma\cong \Gamma$. Therefore $\Gamma$ is 2-step Minkowski if and only if $\Out(\Gamma)$ and $\Aut(\Gamma)$ are Minkowski.
\end{rem}

\begin{prop}\label{2-step Minkowski lattice}
	Let $\Gamma$ be lattice of a connected Lie group, then $\Gamma$ has the 2-step Minkowski property.
\end{prop}

\begin{proof}
	Recall that there is a short exact sequence $1\longrightarrow \Gamma_A\longrightarrow \Gamma\longrightarrow \Gamma_{nc}\longrightarrow 1$ where $\Gamma_A$ is virtually polycyclic and $\Gamma_{nc}$ is a centreless lattice in a semisimple Lie group (see \cite[Section 6]{daura2024actions}). In consequence, $\Zc\Gamma\trianglelefteq \Gamma_A\trianglelefteq \Gamma$. Thus, we have a short exact sequence $1\longrightarrow \Gamma_A/\Zc\Gamma\longrightarrow \Inn\Gamma\longrightarrow \Gamma_{nc}\longrightarrow 1$. Moreover $\Gamma_A/\Zc\Gamma$ is virtually polycyclic.
	
	Since $\Gamma_{nc}$ is centreless, the center $\Zc\Inn\Gamma$ is a subgroup of $\Gamma_A/\Zc\Gamma$. Since $\Gamma_A/\Zc\Gamma$ is virtually polycyclic, $\Zc\Inn\Gamma$ is finitely generated. Moreover, using again that $\Gamma_A/\Zc\Gamma$ is virtually polycyclic and $\Gamma_{nc}$ is a centreless lattice in a semisimple Lie group and that $\Gamma_A/\Zc\Gamma$ is virtually polycyclic we can conclude that $\Out(\Inn\Gamma)$ and $\Aut(\Inn\Gamma)$ are Minkowski (see \cite[Section 6]{daura2024actions} and in particular, \cite[Reamrk 6.7]{daura2024actions}).  
	
\end{proof}

We are ready to prove \cref{main theorem9 intro}.

\begin{proof}[Proof of \cref{main theorem9 intro}]
	The idea of the proof is to reduce the general case to the case where $M$ is an infranilmanifold finitely covered by a $2$-step nilmanifold. We divide the proof in $4$ parts.
	
	\textbf{Part 1. Study of the first step of the iterated actions:} Since $\D_2(M)=(f,b)$ there exist a strictly increasing sequence of prime numbers $\{p_i\}_{i\in\mathbb{N}}$ and a sequence of numbers $\{a_i\}_{i\in\mathbb{N}}$ such that $\{(\Z/p_i^{a_i})^f,(\Z/p_i)^b\}\acts M$ freely. Since $\Out(\pi_1(M))$ is Minkowski, there exists $i_0$ such that for all $i\geq i_0$ the induced group morphism $\psi_i:(\Z/p_i^{a_i})^f\longrightarrow \Out(\pi_1(M))$ is trivial. Thus, we can assume without loss of generality that all $\psi_i$ are trivial. Now, we need the following lemma:
	
	\begin{lem}\label{inner actions aspherical manifolds}\cite[Proposition 3.1.21]{lee2010seifert}
		Let $M$ be a closed aspherical manifold and let $\Zc\pi_1(M)$ be finitely generated. Assume that we have a free action of an abelian group $A$ on $M$ such that $\psi:A\longrightarrow \Out(\pi_1(M))$ is trivial. Then $\Zc\pi_1(M/A)=C_{\pi_1(M/A)}(\pi_1(M))$ and it is an extension of $\Zc\pi_1(M)$ by $A$. In particular, $\rank(\Zc\pi_1(M/A))=\rank(\Zc\pi_1(M))$.
	\end{lem}
	
	
	Let $\tilde{G_i}=\pi_1(M/(\Z/p_i^{a_i})^f)$. By \cref{aspherical manifolds big diagram} and \cref{inner actions aspherical manifolds}, we have the following commutative diagram
	\[\begin{tikzcd}
		& 1\ar{d}{}& 1 \ar{d}{}& & \\
		1\ar{r}{}& \Zc\pi_1(M)\ar{d}{}\ar{r}{}& C_{\tilde{G}}(\pi_1(M))=\Zc\tilde{G}_i \ar{d}{}\ar{r}{}& (\Z/p_i^{a_i})^f \ar{d}{id}\ar{r}{}& 1\\
		1\ar{r}{}& \pi_1(M)\ar{d}{}\ar{r}{}& \tilde{G}_i \ar{d}{}\ar{r}{p}&(\Z/p_i^{a_i})^f \ar{r}{}& 1\\
		& \Inn\pi_1(M)\ar{d}{}\ar{r}{\cong}& \Inn \tilde{G}_i \ar{d}{}& & \\	
		& 1& 1 &  & \\
	\end{tikzcd}\]
	
	where the isomorphism of the last row is induced by the inclusion morphism $\Inn\pi_1(M)\longrightarrow\Aut(\pi_1(M))$. Since $\pi_1(M)$ is 2-step Minkowski and $\Inn\pi_1(M)\cong\Inn \tilde{G}_i$ for all $i$, the groups $\Aut(\Inn \tilde{G}_i)$ and $\Out(\Inn \tilde{G}_i)$ are Minkowski for all $i$. The key observation is that Minkowski constants of $\Aut(\Inn \tilde{G}_i)$ and $\Out(\Inn \tilde{G}_i)$ do not depend on $i$.
	
	\textbf{Part 2. Study of the second step of the iterated actions:} Note that $M/(\Z/p_i^{a_i})^f$ is a closed aspherical manifold. Therefore the second step of each iterated action induces a group morphism $\psi_i':(\Z/p_i)^b\longrightarrow\Out(\tilde{G}_i)$. We claim that $\psi_i'$ is injective for all $i$. Assume the contrary. Then, each iterated action is equivalent to the iterated action in $3$ steps $\{(\Z/p_i^{a_i})^b,\ker\psi_i',(\Z/p_i)^b/\ker\psi_i'\}\acts M$, where the first 2 steps of the iterated action are inner actions. 
	
	\begin{lem}\cite[Lemma 5.15]{daura2025IGG1}\label{inner actions simplifiable}
		Let $\{A,A'\}\acts M$ be a free iterated action of abelian groups on a closed connected aspherical manifold such that $\psi:A\longrightarrow \Out(\pi_1(M))$ and $\psi':A'\longrightarrow\Out(\pi_1(M/A))$ are trivial. Then $\{A,A'\}\acts M$ is simplifiable by an abelian group.
	\end{lem}
		
	
	By \cref{inner actions simplifiable}, the iterated action $\{(\Z/p_i^{a_i})^b,\ker\psi_i'\}\acts M$ is simplifiable by an abelian group $A_i$ for all $i$. This implies that $\{(\Z/p_i^{a_i})^b,(\Z/p_i)^b\}\acts M$ and $\{A_i,(\Z/p_i)^b/\ker\psi_i\}\acts M$ are equivalent. Since $\D_2(M)=(f,b)$ we obtain that $\rank A_i=f$. Moreover, we have $\rank((\Z/p_i)^b/\ker\psi_i)<b$. Consequently, $\rank_{ab}(\{(\Z/p_i^{a_i})^b,(\Z/p_i)^b\}\acts M)<f+b$, which contradicts the fact that $\D_2(M)=(f,b)$. Thus, the only possibility is that $\ker\psi'_i$ is trivial for all $i$.
	
	Like in the first step of the proof, we can use \cref{aspherical manifolds big diagram} and \cref{inner actions aspherical manifolds} to obtain the commutative diagram 
	
	\[\begin{tikzcd}
		& 1\ar{d}{}& 1 \ar{d}{}& & \\
		& \Zc\tilde{G}_i\ar{d}{}\ar{r}{id}& C_{\tilde{G}_i'}(\tilde{G}_i)\ar{d}{}& & \\
		1\ar{r}{}& \tilde{G}_i\ar{d}{}\ar{r}{}& \tilde{G}_i' \ar{d}{}\ar{r}{p}&(\Z/p_i)^b \ar{r}{}\ar{d}{id}& 1\\
		1\ar{r}{}& \Inn\tilde{G}_i\ar{d}{}\ar{r}{}& G_i \ar{d}{}\ar{r}{}& (\Z/p_i)^b\ar{r}{} &1 \\	
		& 1& 1 &  & \\
	\end{tikzcd}\]
	
	where $\tilde{G}_i'$ is the fundamental group of $M/\{(\Z/p_i^{a_i})^f,(\Z/p_i)^b\} $ and $G_i\leq \Aut(\tilde{G}_i)$. The key observation in this case is the following lemma:
	
	\begin{lem}
		The group $\Aut(\tilde{G}_i)$ is Minkowski with a constant that does not depend on $i$.
	\end{lem}
	
	\begin{proof}
		Consider the short exact sequence $\begin{tikzcd} 1\ar{r}{}& \Zc\tilde{G}_i\ar{r}{}& \tilde{G}_i\ar{r}{}& \Inn\tilde{G}_i\ar{r}{}& 1 \end{tikzcd} $. 
		Since $\Zc\tilde{G}_i$ is a characteristic subgroup there exist short exact sequences
		\[\begin{tikzcd} 1\ar{r}{}& K_i\ar{r}{}& \Aut(\tilde{G}_i)\ar{r}{}& L_i\ar{r}{}& 1 \end{tikzcd} \]
		and 
		\[\begin{tikzcd} 1\ar{r}{}& Z^1(\Inn\tilde{G}_i,\Zc\tilde{G}_i)\ar{r}{}& K_i\ar{r}{}& L'_i\ar{r}{}& 1 \end{tikzcd} \]
		such that $L_i\leq \Aut(\Inn\tilde{G}_i)$ and $L_i'\leq \Aut(\Zc\tilde{G}_i)$. Since $\Inn\tilde{G}_i\cong \Inn\pi_1(M)$ and $\Zc\tilde{G}_i\cong \Zc\pi_1(M)$ we obtain that $\Aut(\tilde{G}_i)$ is Minkowski with a constant not depending on $i$.
	\end{proof}
	
	\textbf{Part 3. $M$ is an infra-nilmanifold:} Consider the abstract kernel of the extension \[\begin{tikzcd} 1\ar{r}{}& \Inn\tilde{G}_i\ar{r}{}& G_i\ar{r}{}& (\Z/p_i)^b\ar{r}{}& 1 \end{tikzcd}\] which we denote by $\phi_i:(\Z/p_i)^b\longrightarrow\Out(\Inn\tilde{G}_i )$. Since $\Out(\tilde{G}_i)$ is Minkowski with a constant not depending on $i$, $\phi_i$ is trivial for $i$ large enough. Thus, we will assume that $\phi_i$ is trivial. We obtain the following diagram where the first row is a central extension and the columns are inclusions.
	
	\[\begin{tikzcd} 
		1\ar{r}{}& \Zc\Inn\pi_1(M)\ar{r}{}\ar{d}{}& C_{G_i}(\Inn\pi_1(M))\ar{r}{}\ar{d}{}& (\Z/p_i)^b\ar{r}{}\ar{d}{id}& 1\\
		1\ar{r}{}& \Inn\pi_1(M)\ar{r}{}& G_i\ar{r}{}& (\Z/p_i)^b\ar{r}{}& 1
	\end{tikzcd}\]
	
	Note that we cannot assume that $C_{G_i}(\Inn\pi_1(M))$ or $\Zc\Inn\pi_1(M)$ are torsion-free. Nevertheless, the group $\Zc\Inn\pi_1(M)$ is finitely generated. Consequently, $\Zc\Inn\pi_1(M)\cong\Z^r\oplus T$, where $T$ is the torsion subgroup of $\Zc\Inn\pi_1(M)$. Moreover,  $C_{G_i}(\Inn\pi_1(M))\leq G_i$, which is Minkowski with a constant which does not depend on $i$. Therefore, for $i$ large enough we can assume that the order of torsion elements of $G_i$ is smaller than $p_i$. In this setting we can use the following lemma:
	
	\begin{lem} 
		Let \[\begin{tikzcd} 1\ar{r}{}& \Z^r\oplus T\ar{r}{}& G\ar{r}{}& (\Z/p)^b\ar{r}{}& 1 \end{tikzcd}\] be a central group extension such that the order of the torsion elements of $G$ is smaller than the prime $p$. Then $r\geq b$.
	\end{lem}
	
	\begin{proof}
		We have a central short exact sequence \[\begin{tikzcd} 1\ar{r}{}& \Z^r\ar{r}{}& G/T\ar{r}{}& (\Z/p)^b\ar{r}{}& 1 \end{tikzcd}\] 
		Note that $G/T$ is torsion-free. If not, there would exist an element $g\in G$ such that $g^p\in T$ and therefore its order $o(g)\geq p$, contradicting the fact that the order of the torsion elements of $G$ is smaller than the prime $p$. Since the extension is central and $G/T$ is torsion-free this implies that $r\geq b$.
	\end{proof}
	
	In consequence $\Z^b\leq \Inn\pi_1(M)$ and we have the following commutative diagram where the rows are central extensions and the columns are inclusions:
	
	\[\begin{tikzcd} 
		1\ar{r}{}& \Z^f\ar{r}{}\ar{d}{id}& \Gamma\ar{r}{}\ar{d}{}& \Z^b\ar{r}{}\ar{d}{}& 1\\
		1\ar{r}{}& \Z^f\ar{r}{}& \pi_1(M)\ar{r}{}& \Inn\pi_1(M)\ar{r}{}& 1
	\end{tikzcd}\]
	
	The group $\Gamma$ is a finitely generated torsion-free 2-step nilpotent group and hence it is a lattice in a $2$-step nilpotent Lie group $N$. The nilmanifold $N/\Gamma$ is a the total space of a principal $T^f$-fibration over $T^b$.
	
	Let $\tilde{M}$ denote the universal covering of $M$. We claim that the covering $\tilde{M}/\Gamma\longrightarrow M$ is a finite covering. Indeed, since $\tilde{M}$ is contractible, we have $H^*(\tilde{M}/\Gamma)\cong H^*(\Gamma,\Z)=H^*(N/\Gamma,\Z)$. Therefore, $H^n(\tilde{M}/\Gamma)\neq 0$, which implies that $\tilde{M}/\Gamma$ is a closed connected manifold. The map $\tilde{M}/\Gamma\longrightarrow M$ is a covering between closed manifolds and hence a finite covering.
	
	We reach the conclusion that $[\pi_1(M):\Gamma]< \infty$. Since $\pi_1(M)$ is torsion-free, $\pi_1(M)$ is an almost-Bieberbach group (see \cref{AC-groups charactherization}). Since the Borel conjecture holds for almost-Bieberbach groups (see \cite{bartels2012borel}), we obtain that $M$ is an infra-nilmanifold.
	
	\begin{rem}
		The central short exact sequence \[\begin{tikzcd} 
			1\ar{r}{}& \Z^f\ar{r}{}& \Gamma\ar{r}{}& \Z^b\ar{r}{}& 1
		\end{tikzcd}\]
		can be trivial and therefore $\Gamma\cong \Z^{f+b}=\Z^n$. This implies that $M$ is a flat manifold and $\D_2(M)=(\rank\pi_1(M),0)$ by \cref{iterated discsym flat manifolds}. Thus, $b=0$ and $\rank\Zc\pi_1(M)=f=n$. By \cref{discsym rigidity}.2, we can conclude that $M\cong T^n$, which completes the proof of \cref{main theorem9 intro} in this particular case. Thus, from now on we will assume that the above central exact sequence is not trivial and hence $\Gamma$ will be a 2-step nilpotent torsion-free group. 
	\end{rem}
	
	\textbf{Part 4. If $M$ is an infra-nilmanifold with $\D_2(M)=(f,b)$ and $f+b=n$ then $M$ is a nilmanifold:} To prove the bolded claim and finish the proof of \cref{main theorem9 intro} we need some preliminary results of free group actions on 2-step nilmanifolds.
	
	Let $N/\Gamma$ be a 2-step nilmanifold of dimension $n$. Assume that $\rank\Zc\Gamma=f$ and let $b=n-f$. Thus, $N/\Gamma$ is the total space of principal $T^f$-bundle over $T^b$ and $\Gamma$ is a finitely generated torsion-free 2-step nilpotent group fitting in the short exact sequence $1\longrightarrow \Zc\Gamma\cong \Z^f\longrightarrow \Gamma\longrightarrow \Z^b\longrightarrow 1$. Then we have short exact sequences 
	$$1\longrightarrow K\longrightarrow \Out(\Gamma)\longrightarrow \Out(\Z^b)=\Gl(\Z,b)\longrightarrow 1$$
	and 
	$$1\longrightarrow \overline{H}^{1}(\Z^b,\Z^f)\longrightarrow K\longrightarrow K'\longrightarrow 1,$$
	where $K'\leq \Out(\Z^b)=\Gl(b,\Z)$ and $\overline{H}^1(\Z^b,\Z^f)=\{[g]\in\Out(\Gamma):g_{|\Z^f}=id_{\Z^f},\overline{g}:\Z^b\longrightarrow \Z^b=id_{\Z^b}\}$.
	
	Let $G$ be a finite group acting freely on $N/\Gamma$ and let $\tilde{G}$ be the fundamental group of the manifold $(N/\Gamma)/G$. There exists a group morphism $\phi:G\longrightarrow \Out(\Gamma)$ and therefore we can consider the morphism $\phi_\Q:G\longrightarrow \Out(\Gamma_\Q)$. Our next goal is to study $\ker(\phi_\Q:G\longrightarrow \Out(\Gamma_\Q))$. The first result we need is proved in \cite{daura2024actions}.  
	
	\begin{prop}\label{discsym coverings}\cite[Proposition 3.14]{daura2024actions}
		Let $M$ be a closed connected aspherical manifold such that $\Zc\pi_1(M)$ is finitely generated. Assume that $G$ is a finite group acting freely on $M$. Then $\rank(\Zc\pi_1(M))$ and $\rank(\Zc\pi_1(M/G))$ are equal if and only if the morphism $\psi':G\longrightarrow \Aut(\Zc\pi_1(M))$, obtained by restricting automorphisms of $\pi_1(M)$ on $\Zc\pi_1(M)$, is trivial.
	\end{prop}	
	
	In the setting of \cref{main theorem9 intro}, we can rewrite \cref{discsym coverings} as follows:
	
	\begin{cor}\label{infranilmanifold rank center}
	We have $\rank\Zc\tilde{G}=\rank \Zc\Gamma$ if and only if $\phi':G\longrightarrow \Aut(\Zc\Gamma)$ is trivial. In this situation, $\Zc\tilde{G}=C_{\tilde{G}}(\Gamma)$.
	\end{cor}
	
	Recall that we can define the isolator of the commutator $[\Gamma,\Gamma]$ as $\sqrt{[\Gamma,\Gamma]}=\{\gamma\in\Gamma:\gamma^r\in[\Gamma,\Gamma]\text{ for some $r$}\}$. It is a characteristic subgroup of $\Gamma$ and $\Gamma/\sqrt{[\Gamma,\Gamma]}$ is torsion-free (see \cite[Lemma 1.1.2]{dekimpe2006almost}). 
	
	\begin{lem}\cite[Proposition 2.4.1]{dekimpe2006almost}\label{automorphism queotient isolator subgroup}
		Given a positive integer $a$, consider the extension $$1\longrightarrow \Gamma\longrightarrow \tilde{\Gamma}\longrightarrow \Z/a\longrightarrow 1.$$
		The group $\tilde{\Gamma}$ is nilpotent if and only if the induced map $\overline{\phi}:\Z/a\longrightarrow \Aut(\Gamma/\sqrt{[\Gamma,\Gamma]})$ is trivial. 
	\end{lem}
	
	Since $\Gamma$ is 2-step nilpotent, $\sqrt{[\Gamma,\Gamma]}\subseteq \Zc\Gamma$ and there exists a number $l$ such that $\Zc\Gamma=\sqrt{[\Gamma,\Gamma]}\oplus\Z^l$. Thus, $\Gamma/\sqrt{[\Gamma,\Gamma]}\cong \Gamma/\Zc\Gamma\oplus \Z^l$.
	
	\begin{lem}\label{infranilmanifold H1}
		We have $\phi(\ker\phi_\Q)\leq \overline{H}^1(\Z^b,\Z^f)$. Conversely, if $g\in G$ such that $\phi(g)\in \overline{H}^1(\Z^b,\Z^f)$, then $g\in \ker\phi_\Q$.
	\end{lem}
	
	\begin{proof}
		Let $g\in \ker\phi_\Q$. Then there exists $x\in\Gamma_\Q$ such that $\phi(g)(\gamma)=x\gamma x^{-1}$ for all $\gamma\in\Gamma$. Thus, $\phi(g)_{|\Z^f}=id_{\Z^f}$ since $\Z^f$ is in the center of $\Gamma_\Q$ and $\overline{\phi(g)}=id_{\Z^b}$ since $\overline{\phi(g)}$ is a conjugation on an abelian group.
		
		Conversely, given $g\in G$, we consider the extension $$1\longrightarrow \Gamma\longrightarrow \tilde{\Gamma}\longrightarrow \langle g\rangle\longrightarrow 1.$$ If $\phi(g)\in \overline{H}^1(\Z^b,\Z^f)$, then $\overline{\phi(g)}=id_{\Z^b}$ and $\phi(g)_{\Z^f}=id_{\Z^f}$. In particular, $\phi(g)$ fixes all elements in $\Z^l$. Thus, since $\langle g\rangle$ is a finite group, we have $\overline{\phi}:\langle g\rangle\longrightarrow \Aut(\Gamma/\sqrt{[\Gamma,\Gamma]})=\Aut(\Z^b\oplus\Z^l)$ is trivial. In consequence, $\tilde{\Gamma}$ is a torsion-free nilpotent group and $g\in\ker\phi_\Q$ by \cref{infranilmanifolds malcev completion trivial}.
	\end{proof}
	
	We are ready to finish the proof of \cref{main theorem9 intro}. Recall that if $\D_2(M)=(f,b)$ with $f+b=n$ then there exists a finite index subgroup $\Gamma\leq \pi_1(M)$ which is the fundamental group of a $2$-step nilmanifold $N/\Gamma$, which is the total space of a principal $T^f$-bundle over $T^b$.
	
	We consider the Fitting subgroup $\Fitt(\pi_1(M))$, which is torsion-free, nilpotent and of finite index in $\pi_1(M)$ (see \cref{AC-groups charactherization}). Since $\Gamma$ has finite index on $\pi_1(M)$, the groups $\Fitt(\pi_1(M))$ and $\Gamma$ are commensurable, and therefore we have a central short exact sequence \[\begin{tikzcd} 1\ar{r}{}& \Z^f\ar{r}{}& \Fitt\pi_1(M)\ar{r}{}& \Z^b\ar{r}{}& 1. \end{tikzcd}\]   
	
	In particular, $\D_2(N/\Fitt(\pi_1(M)))=(f,b)$ (see \cite[Theorem 1.16]{daura2025IGG1}).
	
	We also have a short exact sequence  
	\[\begin{tikzcd} 1\ar{r}{}& \Fitt(\pi_1(M))\ar{r}{}& \pi_1(M)\ar{r}{}& G\ar{r}{}& 1 \end{tikzcd}\]
	where $G$ is a finite group. If $\phi:G\longrightarrow \Out(\Fitt(\pi_1(M)))$ denotes the abstract kernel of the group extension, then the induced map $\phi_\Q:G\longrightarrow \Out(\Fitt(\pi_1(M))_\Q)$ on the rational Mal'cev completion is injective (see \cref{infranilmanifolds malcev completion injective}).
	
	Since $M$ is an infranilmanifold, we have $\D(M)=\rank\Zc\pi_1(M)$ (see \cite{daura2024actions} and \cite[Theorem 11.7.7]{lee2010seifert}). In consequence, $\rank\Zc\pi_1(M)=\rank\Zc\Fitt(\pi_1(M))=f$ and $\phi'$ is trivial by \cref{infranilmanifold rank center}. The injectivity of $\phi_\Q$ together with \cref{infranilmanifold H1} implies that $\overline{\phi}:G\longrightarrow \Out(\Fitt(\pi_1(M))/\Zc\Fitt(\pi_1(M)))\cong \Gl(b,\Z)$ is injective. We obtain the commutative diagram (see \cref{aspherical manifolds big diagram})
	
	\[\begin{tikzcd}
		& 1\ar{d}{}& 1 \ar{d}{}& & \\
		& \Zc\Fitt(\pi_1(M))=\Z^f\ar{d}{}\ar{r}{id}& \Zc\pi_1(M)\ar{d}{}& & \\
		1\ar{r}{}&\Fitt(\pi_1(M))\ar{d}{}\ar{r}{}&\pi_1(M) \ar{d}{}\ar{r}{p}&G \ar{r}{}\ar{d}{id}& 1\\
		1\ar{r}{}& \Inn\Fitt(\pi_1(M))=\Z^b\ar{d}{}\ar{r}{}& \Inn\pi_1(M) \ar{d}{}\ar{r}{}& G\ar{r}{} &1 \\	
		& 1& 1 &  & \\
	\end{tikzcd}\]
	where the identity map in the first row is induced by the inclusion morphism $\Zc\Fitt(\pi_1(M))\longrightarrow C_{\pi_1(M)}(\Fitt(\pi_1(M)))=\Zc\pi_1(M)$. The abstract kernel of the third row is $\overline{\phi}$. Since $\overline{\phi}$ is injective, $\Inn\pi_1(M)$ is a crystallographic group. Moreover, $\D_2(M)=(f,b)$ implies that $\Z^b\leq\Zc\Inn\pi_1(M)$ and hence $[\Inn\pi_1(M):\Zc\Inn\pi_1(M)]<\infty$. Therefore, $[\Inn\pi_1(M),\Inn\pi_1(M)]$ is a finite normal subgroup of a crystallographic group, which implies that $[\Inn\pi_1(M),\Inn\pi_1(M)]$ is trivial (see \cref{AC-groups charactherization}). Hence $\Inn\pi_1(M)$ is abelian and $\pi_1(M)$ is a finitely generated torsion-free 2-step nilpotent group. This implies that $M$ is a nilmanifold, as desired.
\end{proof}

The same proof can be used to obtain the following bound on the iterated discrete degree of symmetry.

\begin{cor}\label{discsym2 bound}
	Let $M$ be a closed connected aspherical manifold such that $\pi_1(M)$ is 2-step Minkowski. Then $\D_2(M)\leq (\rank\Zc\pi_1(M),\rank\Zc\Inn\pi_1(M))$ and $\rank\Zc\pi_1(M)+\rank\Zc\Inn\pi_1(M)\leq \dim(M)$.
\end{cor}

Combining \cref{2-step Minkowski lattice} and \cref{discsym2 bound}, we obtain:

\begin{cor}
	Let $G$ be a connected Lie group, $K$ a maximal subgroup of $G$ and $\Gamma$ a torsion-free cocomapct lattice of $G$. The closed aspherical locally homogeneous space $\Gamma\setminus G/K$ satisfies $\D_2(M)\leq (\rank\Zc\Gamma,\rank\Zc(\Inn\Gamma))$.
\end{cor}

\Cref{discsym2 bound} can be used in the following situation:

\begin{cor}
	Let $M$ be a closed connected aspherical manifold. Suppose that the groups $\Out(\pi_1(M))$ and $\Aut(\pi_1(M))$ are Minkowski and that $\Zc\pi_1(M)$ is trivial. Let $E$ be the total space of a principal $T^f$-bundle over $M$. Then $\D_2(E)=(f,0)$. 
\end{cor}

\begin{proof}
	Consider the central short exact sequence
	\[\begin{tikzcd} 1\ar{r}{}& \Zc\pi_1(E)\cong\Z^f\ar{r}{}& \pi_1(E)\ar{r}{}& \pi_1(M)\ar{r}{}& 1. \end{tikzcd}\]
	Consequently, $\Inn\pi_1(E)\cong\pi_1(M)$ and $\pi_1(E)$ is 2-step Minkowski. Hence, by \cref{discsym2 bound}, $\D_2(M)\leq (f,0)$. 
	
	On the other hand, since $E$ has a free action of $T^f$, $E$ also admits free actions of $(\Z/p^{a})^f$ for any prime $p$ and positive integer $a$. Consequently, $\D_2(E)\geq (f,0)$. Thus, $\D_2(E)=(f,0)$.
\end{proof}

\section{Free iterated actions on closed aspherical 3-dimensional manifolds}\label{sec: iterated actions on 3-manifolds}

If $M$ is a closed $3$-dimensional aspherical manifold with an effective $S^1$ action, then $M$ can be one of the following four cases (see \cite[\S 14.4]{lee2010seifert}):
\begin{itemize}
	\item[1.] $M\cong T^3$.
	\item[2.] $M$ is homeomorphic to $K\times S^1$ or $SK$, where $K$ denotes the Klein bottle and $SK$ the non-trivial principal $S^1$-bundle over $K$.
	\item[3.] $M\cong H/\Gamma$, where $H$ is the $3$-dimensional Heisenberg group and $\Gamma $ is a lattice of $H$
	\item[4.] $\Zc\pi_1(M)\cong \Z$, and $\Inn \pi_1(M)\cong \pi_1(M)/ \Zc\pi_1(M)$ is centreless.
\end{itemize}

Note that in all cases, we have a central extension $$1\longrightarrow \Z\longrightarrow \pi_1(M)\longrightarrow Q\longrightarrow 1$$ where $Q$ acts effectively, properly and cocompactly on $\R^2$. If $M$ is a closed connected aspherical 3-manifold then $\Out(\pi_1(M))$ is Minkowski by \cite{kojima1984bounding} and $\A(M)=\D(M)=\rank\Zc\pi_1(M)$ by \cite[Corollary 8.3]{gabai1992convergence} and \cite[Theorem 1.1]{casson1994convergence}. Thus, it is straightforward to show that $\D(T^3)=3$, $\D(K\times S^1)=\D(SK)=2$, $\D(H/\Gamma)=1$, and $\D(M)=1$ if $M$ belongs to the case 4. Moreover if $M$ is a closed connected aspherical 3-manifold which does not belong to one of the four cases above, then $\A(M)=\rank\Zc\pi_1(M)=0$ and hence $\D(M)=0$.

We will compute the iterated discrete degree of symmetry and show that it can be used to distinguish the four cases of the above classification. We start by providing a different proof that $\Out(\pi_1(M))$ is Minkowski if $M$ belongs to one of these four cases.

\begin{lem}\label{out minkowski 3-manifolds}
	Let $M$ be a closed $3$-dimensional aspherical manifold with an effective $S^1$-action. Then $\Zc\pi_1(M)$ is finitely generated and $\Out(\pi_1(M))$ is Minkowski.
\end{lem}
\begin{proof}
	In the first three cases the fundamental group is polycyclic and therefore $\Out(\pi_1(M))$ is Minkowski (this is a consequence of Minkowski's lemma and \cite{Wehrfritz1994}). We only need to check the case where we have a short exact sequence $1\longrightarrow \Zc\pi_1(M)\cong\Z\longrightarrow \pi_1(M)\longrightarrow \Inn(\pi_1(M))\longrightarrow 1$ and $ \Inn(\pi_1(M))$ is centreless. Since $\Zc\pi_1(M)$ is a characteristic subgroup of $\pi_1(M)$ then there are short exact sequences $$1\longrightarrow K \longrightarrow \Out (\pi_1(M))\longrightarrow \Out(\Inn\pi_1(M))\longrightarrow 1$$ and $$1\longrightarrow \overline{H}^1(\Inn\pi_1(M),\Z) \longrightarrow K\longrightarrow \Gl(1,\Z)\longrightarrow 1.$$
	
	Recall that $\overline{H}^1(\Inn\pi_1(M),\Z)\cong Z^1(\Inn\pi_1(M),\Z)/\overline{B}^1(\Inn\pi_1(M),\Z)$ and $$\overline{B}^1(\Inn\pi_1(M),\Z)\cong p^{-1}(\Zc\Inn\pi_1(M)\cap C_{\pi_1(M)}(\Zc\pi_1(M)))/\Zc\pi_1(M)$$ where $p:\pi_1(M)\longrightarrow \Inn\pi_1(M)$ is the quotient map. Since $\Zc\Inn\pi_1(M)$ is trivial, the group $\overline{B}^1(\Inn\pi_1(M),\Z)$ is also trivial and $$\overline{H}^1(\Inn\pi_1(M),\Z)\cong Z^1(\Inn\pi_1(M),\Z)\cong H^1(\Inn\pi_1(M),\Z).$$ 
	
	The group $\Out(\pi_1(M))$ will be Minkowski if $\Out(\Inn\pi_1(M))$ and $H^1(\Inn\pi_1(M),\Z)$ are Minkowski. Since $\Inn\pi_1(M)$ acts effectively, properly and cocompactly on $\R^2$, $\Inn\pi_1(M)$ is a subgroup of isometries of the Euclidean plane or the hyperbolic plane. In both cases $H^1(\Inn\pi_1(M),\Z)$ is a finitely generated abelian group and therefore Minkowski. 
	
	If $\Inn\pi_1(M)$ is a subgroup of isometries of the Euclidean plane then it is virtually abelian and therefore $\Out(\Inn\pi_1(M))$ is Minkowski. 
	If $\Inn\pi_1(M)$ is a subgroup of isometries of the hyperbolic plane, then it contains a centreless torsion-free Fuchsian subgroup $Q$ of finite index. Since $\Out(Q)$ is virtually torsion-free, $\Out(\Inn\pi_1(M))$ is also virtually torsion-free (see \cite[Lemma 2.4, Corollary 2.6
	]{metaftsis2006residual}) and therefore Minkowski.
\end{proof}

The proof of \cref{main theorem10 intro} is a combination of \cref{inner actions simplifiable}, \cref{inner actions aspherical manifolds} and the following lemma.

\begin{lem}\label{discsym2 finiteness}
	Let $M$ be a closed aspherical manifold. Assume that $\Zc\pi_1(M)$ is finitely generated, that $\Out(\pi_1(M))$ is Minkowski and that there exists a constant $C$ satisfying that for every free inner action of an abelian group $A$ on $M$, $\Out(\pi_1(M/A))$ is Minkowski with a constant less or equal than $C$. Then $\D_2(M)=(d_1,0)$.
\end{lem}

\begin{proof}
	Assume that $\D_2(M)=(d_1,d_2)$. We have an increasing sequence of primes $\{p_i\}_{i\in \mathbb{N}}$, a sequence of positive integers $\{a_i\}_{i\in \mathbb{N}}$ and a collection of free iterated actions $\{(\Z/p_i^{a_i})^{d_1},(\Z/p_i)^{d_2}\}\acts M$ satisfying that $$\rank_{ab}(\{(\Z/p_i^{a_i})^{d_1},(\Z/p_i)^{d_2}\}\acts M)=(d_1,d_2).$$ We can assume that $p_i>C$ for all $i$. In this case, $\psi_i:(\Z/p_i^{a_i})^{d_1}\longrightarrow \Out(\pi_1(M))$ and $\psi_i':(\Z/p_i)^{d_2}\longrightarrow \Out(\pi_1(M))$ are trivial for all $i$. By \cref{inner actions simplifiable}, this iterated action is simplifiable, which implies that there exists a free group action of $(\Z/p_i)^{d_1+d_2}$ on $M$. But then $d_1+d_2\leq d_1$, which implies that $d_2=0$.
\end{proof}

From \cref{discsym2 finiteness} we can deduce the following corollary, which proves part (2) of \cref{main theorem10 intro}. 

\begin{cor}\label{iterated discsym flat manifolds}
	Let $M$ be a closed flat manifold. Then $\D_2(M)=(\rank\Zc\pi_1(M),0)$.
\end{cor}

\begin{proof}
	Given any abelian group $A$ acting freely on a $n$-dimensional closed flat manifold $M$ such that $\psi:A\longrightarrow \Out(\pi_1(M))$, the quotient space $M/A$ is a closed flat manifold of the same dimension and hence $\pi_1(M/A)$ is a Bieberbach group. Because there is a finite number of isomorphism classes of Bieberbach groups in each dimension and their outer automorphisms group is Minkowski, we can use \cref{discsym2 finiteness} by taking $C$ to be the maximum of all Minkowski constants of the outer automorphism group of Bieberbach groups of flat manifolds of dimension $n$. 
\end{proof}

\begin{proof}[Proof of \cref{main theorem10 intro}]
	By \cref{main theorem9 intro}, we have already proved that $\D_2(M)=(3,0)$ if $M\cong T^3$ and that $\D_2(M)=(1,2)$ if $M\cong H/\Gamma$, as a consequence \cref{main theorem9 intro}. Moreover, $SK$ and $K\times S^1$ are closed flat manifolds, which implies that $\D_2(K\times S^1)=\D_2(SK)=(2,0)$ (see \cref{iterated discsym flat manifolds}). Therefore, we only need to check the last case.
	
	Let $M$ be a $3$-aspherical manifold of the case 4 and assume that the Minkowski constant of $\Out(\pi_1(M))$ is $C$. Since $\D(M)=\A(M)=\rank\Zc\pi_1(M)=1$, there exist an increasing sequence of primes $\{p_i\}_{i\in\mathbb{N}}$ and a sequence of positive integers $\{a_i\}_{i\in\mathbb{N}}$ such that $\Z/p_i^{a_i}$ acts freely on $M$ and $p_i>C$ for all $i$. Therefore each action induces the trivial group morphism $\psi_i:\Z/p_i^{a_i}\longrightarrow \Out(\pi_1(M))$ and a commutative diagram
	\[
	\begin{tikzcd}
		& 1\ar{d}{}& 1\ar{d}{}& & \\
		1\ar{r}{} & \Zc \pi_1(M)\cong \Z \ar{r}{}\ar{d}{} & C_{G_i}(\pi_1(M))=\Zc G_i\ar{r}{}\ar{d}{}& \Z/p_i^{a_i}\ar{r}{}\ar{d}{Id} &1\\
		1\ar{r}{} & \pi_1(M) \ar{r}{}\ar{d}{} & G_i\ar{r}{}\ar{d}{}& \Z/p_i^{a_i}\ar{r}{} &1\\
		& \Inn \pi_1(M)\ar{d}{}\ar{r}{Id} & \Inn \pi_1(M)\ar{d}{} & &\\
		& 1 &1 & & \\
	\end{tikzcd}
	\]
	where $G_i=\pi_1(M/(\Z/p_i^{a_i}))$. Note that $\Inn\pi_1(M)$ is centreless. 
	
	We claim that $\Out(G_i)$ is Minkowski with a constant that does not depend on $i$. Since $\Zc G_i$ is a characteristic subgroup of $G_i$ and $\Inn\pi_1(M)$ is centreless there exist two short exact sequence (see \cref{out minkowski 3-manifolds}) $$1\longrightarrow K \longrightarrow \Out (G_i)\longrightarrow \Out(\Inn\pi_1)\longrightarrow 1$$ and $$1\longrightarrow H^1(\Inn\pi_1(M),\Z) \longrightarrow K\longrightarrow \Gl(1,\Z)\longrightarrow 1.$$
	
	The Minkowski constant of the group $\Out(G_i)$ can be bounded by the Minkowski constants of $\Out(\Inn\pi_1(M))$, $\Gl(1,\Z)$ and $H^1(\Inn\pi_1(M),\Z)$, which do not depend on $p_i$. Thus, we can use \cref{discsym2 finiteness} to conclude that $\D_2(M)=(1,0)$.
\end{proof}

\begin{rem}
	The key observation of the proof of \cref{main theorem10 intro} is that if $\Inn\pi_1(M)$ is centreless then we have $\bar{H}^1(\Inn\pi_1(M),\Z)=H^1(\Inn\pi_1(M),\Z)$, which does not depend on the action of $(\Z/p_i^{a_i})$ on $M$. In general, $\bar{H}^1(\Inn\pi_1(M),\Z)$ does depend on the action of $(\Z/p_i^{a_i})$ on $M$.
\end{rem}

\section{The length of a free iterated action}\label{sec:lenght iterated actions}

A natural question is whether for a closed manifold $M$, there exists a constant $C$ such that $l(M)\leq C$. For example:

\begin{lem}\label{free iterated action lenght Euler char}
	Let $M$ be a closed manifold with $\chi(M)\neq 0$. Then $l(M)\leq \log_2\chi(M)$.
\end{lem}

\begin{proof}
	Recall that if $G$ is a finite group acting freely on $M$ then $\chi(M)=|G|\chi(M/G)$. If we have an iterated free action $\mathcal{G}\acts M$ then $\chi(M)=\prod_{i=1}^n |G_i| \chi(M/\mathcal{G})$. Since $|G_i|\geq 2$ for all $i$, then $n\leq \log_2\chi(M)$.
\end{proof}

The first main result of this section bounds the iterated length of nilmanifolds. Let $N/\Gamma$ be a $c$-step nilmanifold and let $\{0\}=Z_0\trianglelefteq Z_1\trianglelefteq\cdots\trianglelefteq Z_c=\Gamma$ be the upper central series. The groups $Z_{i+1}/Z_i$ are finitely generated torsion-free and abelian, so we denote by $b_i=\rank(Z_{i+1}/Z_i)$. Then:

\begin{thm}\label{free iterated actions on nilmanifolds}
	There exists a constant $C$ only depending on $\Gamma$ such that any free iterated action $\mathcal{G}\acts N/\Gamma$ is equivalent to a free iterated action $\mathcal{G}'\acts N/\Gamma$ where $\mathcal{G}'=\{A_1,\dots,A_c,G'\}$, $A_i$ are finite abelian groups such that $\rank(A_i)\leq b_i$ and $|G'|\leq C$.
\end{thm}

\begin{cor}\label{free iterated actions on nilmanifolds lenght}
	If $N/\Gamma$ is a $c$-step nilmanifold then $l(N/\Gamma)\leq c+1$.
\end{cor}

We need some preliminary lemmas before proving \cref{free iterated actions on nilmanifolds}. Firstly, we study the case where $(N/\Gamma)/\mathcal{G}$ is a nilmanifold.

\begin{lem}\label{free iterated actions nilmanifolds quotient nilmanifold}
	Let $N$ be a $c$-step nilpotent Lie group and let $\Gamma$ and $\Lambda$ be lattices of $N$ such that $\Gamma\leq \Lambda$. There exist subgroups $\Lambda_0,\Lambda_1,\dots,\Lambda_c$ such that $\Lambda_0=\Gamma$, $\Lambda_c=\Lambda$ and $\Lambda_i\trianglelefteq \Lambda_{i+1}$ for all $i$. Moreover, $A_i=\Lambda_{i+1}/\Lambda_{i}$ is abelian and $\rank(A_i)\leq b_i$ for all $i$.
\end{lem}

\begin{proof}
	We prove the claim by induction on $c$. If $c=1$ then $N$ is abelian and therefore $\Gamma=\Lambda_0\trianglelefteq\Lambda_1=\Lambda$.
	
	Assume that $N$ is $c$-step nilpotent and let $\pi:N\longrightarrow N/\Zc N$ be the quotient map to the $(c-1)$-step nilpotent Lie group $N/\Zc N$. Then $\pi(\Lambda)$ and $\pi(\Gamma)$ are lattices of $N/\Zc N$. Since $\Zc\Lambda=\Lambda\cap\Zc N$ and $\Zc\Gamma=\Gamma\cap\Zc N$, there are short exact sequences
	\[
	\begin{tikzcd}
		1 \ar{r}{}& \Zc\Lambda \ar{r}{}&\Lambda \ar{r}{\pi_{|\Lambda}}& \pi(\Lambda) \ar{r}{}&1
	\end{tikzcd}
	\]
	
	and 
	
	\[
	\begin{tikzcd}
		1 \ar{r}{}& \Zc\Gamma \ar{r}{}&\Gamma \ar{r}{\pi_{|\Gamma}}& \pi(\Gamma) \ar{r}{}&1.
	\end{tikzcd}
	\]
	
	By induction hypothesis, there exists $\Lambda'_0=\pi(\Gamma)\trianglelefteq \Lambda'_1\trianglelefteq \dots\trianglelefteq \Lambda'_{c-1}=\pi(\Lambda)$. Then we define $\Lambda_{i+1}=\pi_{|\Lambda}^{-1}(\Lambda'_i)$ for $0\leq i\leq c-1$, which satisfy that $\Lambda_{i}\trianglelefteq \Lambda_{i+1}$ for all $0\leq i\leq c-1$.
	
	Lastly, we take $\Lambda_0=\Gamma$. We have the commutative diagram where the rows are central extensions
	
	\[
	\begin{tikzcd}
		1 \ar{r}{}& \Zc\Lambda \ar{r}{}&\Lambda_1 \ar{r}{\pi_{|\Lambda}}& \pi(\Gamma) \ar{r}{}&1\\
		1 \ar{r}{}& \Zc\Gamma \ar{r}{}\ar[hook]{u}{}&\Gamma \ar{r}{\pi_{|\Gamma}}\ar[hook]{u}{}& \pi(\Gamma) \ar{r}{}\ar{u}{id}&1.
	\end{tikzcd}
	\]
	Consequently, $\Gamma\trianglelefteq \Lambda_1$.
\end{proof}

\begin{cor}
	A free iterated action $\mathcal{G}\acts N/\Gamma$ such that $(N/\Gamma)/\mathcal{G}$ is a nilmanifold is equivalent to a free iterated action $\mathcal{G}'\acts N/\Gamma$ where $\mathcal{G}'=\{A_1,\dots,A_c\}$ and $A_i$ are finite abelian groups such that $\rank(A_i)\leq b_i$.
\end{cor}

For the general case we need \cref{AC-groups charactherization}, \cref{infranilmanifolds malcev completion trivial} to prove the following two lemmas.




\begin{lem}\label{itertaed actions on nilmanifolds and Fitt}
	Let $\mathcal{G}\acts N/\Gamma$ be a free iterated action. Then $\pi_1((N/\Gamma)/\mathcal{G})=E$ is an AC-group and $\Gamma\leq \Fitt(E)$. 
\end{lem}

\begin{proof}
	Since $(N/\Gamma)/\mathcal{G}$ is a closed aspherical manifold, then $E$ is torsion-free and contains $\Gamma$ as a finite index subgroup. Using \cref{AC-groups charactherization}, $E$ is an AC-group. It only remains to prove that $\Gamma\leq \Fitt(E)$.
	
	Consider the exact sequence
	\[
	\begin{tikzcd}
		1\ar{r}{} & \Fitt(E)\ar{r}{} & E\ar{r}{p} & F\ar{r}{} & 1
	\end{tikzcd}
	\]
	If $G=p(\Gamma)$ and $\Lambda=p^{-1}(G)$ we obtain the commutative diagram
	\[
	\begin{tikzcd}
		1\ar{r}{} & \Fitt(E)\cap\Gamma\ar{r}{}\ar{d}{} & \Gamma\ar{r}{p}\ar{d}{i} & G\ar{r}{}\ar{d}{} & 1\\
		1\ar{r}{} & \Fitt(E)\ar{r}{} & \Lambda\ar{r}{p} & G\ar{r}{} & 1
	\end{tikzcd}
	\]
	where the vertical arrows are inclusions. Let $\sigma:G\longrightarrow \Gamma$ be a set-theoretic section of $p$ and $\psi:G\longrightarrow \Out(\Fitt(E)\cap \Gamma)$ the induced group morphism such that $\psi(g)=[c_{\sigma(g)|\Fitt(E)\cap \Gamma}]$. The map $\overline{\sigma}=i\circ \sigma$ is a section for the second short exact sequence and it induces a group morphism $\bar\psi:G\longrightarrow \Out(\Fitt(E))$ such that $\tilde\psi(g)=[c_{\sigma(g)|\Fitt(E)}]$.
	
	Since $\Gamma$ and $\Fitt(E)$ have finite index inside $E$, then $\Fitt(E)\cap \Gamma$ has finite index inside $\Fitt(E)$ and $\Gamma$ and therefore $\Gamma_\Q=\Fitt(E)_\Q=(\Fitt(E)\cap \Gamma)_\Q$. This implies that the maps $\psi':G\longrightarrow \Out(\Gamma_\Q)$ and $\bar\psi':G\longrightarrow \Out(\Gamma_\Q)$ are the same. By \cref{infranilmanifolds malcev completion trivial} the morphism $\psi'$ is trivial since $\Gamma$ is nilpotent. On the other hand $\overline{\psi}'$ is injective by \cref{infranilmanifolds malcev completion trivial} and part 2. of \cref{AC-groups charactherization}. Therefore, the only option is that $G$ is trivial and $\Gamma\leq\Fitt(E)$.
\end{proof}

\begin{lem}\label{out rational malcev completion}
	There exists a constant $C$ such that if $G$ is a finite subgroup of $\Out(\Gamma_\Q)$ then $|G|\leq C$.
\end{lem}

\begin{proof}
	Let $\psi:G\longrightarrow \Out(\Gamma_\Q)$ denote the inclusion of the finite group. We have $H^2_\psi(G,\Q^n)=0$, since $G$ is finite and $\Q^n$ is divisible and therefore there is an injective lift $\tilde{\psi}:G\longrightarrow \Aut(\Gamma_\Q)$ of $\psi$. By \cref{automorphisms malcev completion}, $\Aut(\Gamma_\Q)=\Aut(\mathcal{L}(\Gamma_\Q))$, which is a subgroup of $\Gl(m,\Q)$ for some $m$. In consequence, $G$ is conjugated to a finite subgroup of $\Gl(m,\Z)$ and therefore the bound is a consequence of Minkowski's lemma. 
\end{proof}

\begin{proof}[Proof of \cref{free iterated actions on nilmanifolds}]
	Assume that we have a free iterated action of $\mathcal{G}$ on a nilmanifold $N/\Gamma$. Then $(N/\Gamma)/\mathcal{G}=M$ is a closed almost-flat manifold and $\pi_1(M)$ is an almost-Bieberbach group. In consequence, $\pi_1(M)$ contains a maximal normal nilpotent group $\Lambda$, which has $\Gamma$ as a subgroup by maximality by \cref{itertaed actions on nilmanifolds and Fitt}. Since $[\pi_1(M):\Gamma]<\infty$, $[\Lambda:\Gamma]<\infty$ and therefore $\Lambda$ is a lattice of $N$. By \cref{free iterated actions nilmanifolds quotient nilmanifold}, there exists a subnormal series $\Gamma=\Lambda_0\trianglelefteq\Lambda_1\trianglelefteq\cdots\trianglelefteq\Lambda_c=\Lambda\trianglelefteq \pi_1(M)$. We have a free iterated action of $\{\Lambda_i/\Lambda_{i-1},\pi_1(M)/\Lambda\}_{i=1,\dots,c}$ on $N/\Gamma$ equivalent to $\mathcal{G}\acts N/\Gamma$. By \cref{free iterated actions nilmanifolds quotient nilmanifold}, $\rank(\Lambda_i/\Lambda_{i-1})\leq b_i$. Moreover, $\pi_1(M)/\Lambda\leq \Out(\Lambda_\Q)$ and $\Lambda_\Q\cong \Gamma_\Q$. Therefore, by \cref{infranilmanifolds malcev completion trivial} $|\pi_1(M)/\Lambda|\leq C$, where $C$ is a constant depending on $\Gamma_\Q$.
\end{proof}  

\begin{rem}\label{lenght remark}
	The bound of \cref{free iterated actions on nilmanifolds lenght} is sharp. For example, the free iterated action on $T^3$ from \cite{daura2025IGG1} shows that $2\geq l(T^3)$ and therefore $l(T^3)=2$. Consequently, $l(T^n)=2$ for all $n\geq 3$. On the other hand, in \cite{daura2025IGG1} we show that all free iterated actions on $S^1$ and $T^2$ are simplifiable, which implies that $l(T^n)=1$ for $n=1,2$. Thus, \cref{free iterated actions on nilmanifolds lenght} is not an equality in general. 
\end{rem}

On the other hand, a bound for $l(M)$ when $M$ is a solvmanifold does not always exist.

\begin{thm}\label{solvmanifold inifite lenght}
	There exists a $3$-dimensional solvmanifold $M$ such that $l(M)=\infty$.
\end{thm}

\begin{proof}
	Consider the 3-dimensional solvable Lie group $R=\operatorname{Sol}^3$. Explicitly, $R\cong \R^2\rtimes_\psi\R$, where 
	\[ \psi(t)= \left(\begin{matrix}
		e^{-t} & 0\\
		0 & e^t
	\end{matrix}\right).\] 
	
	Any lattice of $R$ is isomorphic to a semi-direct $\Z^2\rtimes_\phi \Z$, where $\phi:\Z\longrightarrow \Sl(2,\Z)$ satisfies $\operatorname{tr}(\phi(1))>2$ (see \cite[\S 2]{lee2015infra}). Two of these lattices $\Gamma$ and $\Gamma'$ are isomorphic if and only if $\phi'(1)$ is conjugate to $\phi(1)$ or $\phi(1)^t$.
	
	We take the matrix 
	\[ A= \left(\begin{matrix}
		5 & 2\\
		2 & 1
	\end{matrix}\right)\]
	and the lattice $\Gamma=\Z^2\rtimes_\phi \Z$ of $R$ where $\phi(1)=A$. We are going to see that the solvmanifold $M=R/\Gamma$ satisfies $l(M)=\infty$. We start by studying the lattice $\Gamma$ and some of its sublattices.
	
	For an integer $k\geq 0$ we define the group $\Gamma_k=2^k(\Z)^2\rtimes_\phi\Z$. We have a series $\Gamma=\Gamma_0\geq \Gamma_1\geq \Gamma_2 \geq \dots$. Firstly, we note that the group morphism $f_k:\Gamma\longrightarrow \Gamma_k$ such that $f_k(v,t)=(2^kv,t)$ is an isomorphism for all $k\geq 0$, which implies that $R/\Gamma$ and $R/\Gamma_k$ are diffeomorphic for all $k\geq 0$.
	
	We claim now that the normalizer $N_\Gamma(\Gamma_k)=\Gamma_{k-1}$. Firstly, a computation shows that:
	\begin{itemize}
		\item[1.] If $(v,t)\in\Gamma$ then $(v,t)^{-1}=(-A^{-t}v,-t)$.
		\item[2.] If $(v,t),(w,s)\in\Gamma$ then $(v,t)(w,s)(v,t)^{-1}=((Id-A^s)v+A^tw,s)$.
	\end{itemize}
	
	Therefore, the normalizer takes the form 
	$$N_\Gamma(\Gamma_k)=\{(v,t)\in\Gamma:(Id-A^s)v\in 2^{k}\Z^2 \text{ for all $s$}\}.$$
	
	The matrix $A$ is of the form $A=Id+B$ with 
	\[ B= \left(\begin{matrix}
		4 & 2\\
		2 & 0
	\end{matrix}\right)\in M_{2\times 2}(2\Z)\]
	which implies that $Id-A^s=-\sum_{i=1}^{s}\binom{s}{i}B^s\in M_{2\times 2}(2\Z)$ for all $s\geq1$ (and we have a similar form for $A^{-s}$). Let $(v,t)\in N_\Gamma(\Gamma_k)$ with $v=(v_1,v_2)\in \Z^2$. Since $(Id-A^s)v\in 2^{k}\Z^2$ for all $s$, in particular $-Bv=(Id-A)v\in 2^{k}\Z^2$. We obtain that $2v_1,2v_2\in 2^k\Z$ and therefore $v_1,v_2\in 2^{k-1}\Z$. Moreover, if $v\in 2^{k-1}\Z^2$ then $(Id-A^s)v\in 2^{k}\Z^2$ for all $s$. Thus, we have seen that  
	$$N_\Gamma(\Gamma_k)=\{(v,t)\in\Gamma:v\in 2^{k-1}\Z^2\}= \Gamma_{k-1}.$$
	
	In consequence, $\Gamma_k\trianglelefteq\Gamma_{k-1}$ and $\Gamma_k\ntrianglelefteq\Gamma_{k-i}$ for $i>1$. In addition, $ \Gamma_{k-1}/\Gamma_{k}\cong \Z/2\oplus\Z/2$. If $\pi_k:\Gamma_{k-1}\longrightarrow \Z/2\oplus\Z/2$ denotes the quotient map, then we can define new lattices
	$$\Gamma_k^{(1,0)}=\pi_k^{-1}(\langle(1,0)\rangle)=(2^k\Z\times 2^{k-1}\Z)\rtimes_\phi\Z, $$
	$$\Gamma_k^{(0,1)}=\pi_k^{-1}(\langle(0,1)\rangle)=(2^{k-1}\Z\times 2^{k}\Z)\rtimes_\phi\Z, $$
	$$\Gamma_k^{(1,1)}=\pi_k^{-1}(\langle(1,1)\rangle)= \{((v_1,v_2),t)\in\Gamma_{k-1}:v_1+v_2\in 2^{k}\Z\}.$$
	
	Analogous computations show that $N_\Gamma(\Gamma_k^{(i,j)})=\Gamma_{k-1}^{(i,j)}$ and $\Gamma_k^{(i,j)}/\Gamma_{k-1}^{(i,j)}\cong \Z/2\oplus\Z/2$ for all $(i,j)\in \Z/2\oplus\Z/2$, where we set $\Gamma_k^{(0,0)}=\Gamma_{k-1}$.
	
	We consider the regular $\Z/2\oplus \Z/2$ self-covering $p:R/\Gamma_1\cong M\longrightarrow R/\Gamma\cong M$. We can obtain a tower of regular self-coverings $p^k: R/\Gamma_k\cong M\longrightarrow R/\Gamma\cong M$, which has associated a free iterated action of $\mathcal{G}_k=\{\Z/2\oplus\Z/2,\dots,\Z/2\oplus\Z/2\}$ with $l(\mathcal{G}_k)=k$ on $M$. 
	
	Let $\mathcal{G}'=\{G_1',\dots,G_m'\}\acts M$ be a free iterated action equivalent to $\mathcal{G}_k\acts M$. Then we have a subnormal series $\Gamma_k=\Lambda_0\trianglelefteq\Lambda_1\trianglelefteq\cdots\trianglelefteq \Lambda_{m}=\Gamma$ such that $\Lambda_{i}/\Lambda_{i-1}\cong G_i'$. We have $\Lambda_1\leq N_\Gamma(\Gamma_k)$ and therefore $\Lambda_1= \Gamma_k^{(i,j)}$ for some $(i,j)\in\Z/2\oplus\Z/2$. In consequence, $|G_1'|\leq 4$. We can repeat the same process with $\Lambda_1= \Gamma_k^{(i,j)}$, since $N_\Gamma(\Gamma_k^{(i,j)})=\Gamma_{k-1}^{(i,j)}$. Repeating this process we find that all the subgroups of the subnormal series are of the form  $\Gamma_a^{(i,j)}$. This implies that $|G_i'|\leq 4$ for all $1\leq i\leq m$. 
	
	Finally, if $\mathcal{G}'\acts M$ satisfies that $m=l(\mathcal{G}_k\acts M)$ then $[\Gamma:\Gamma_k]=4^k=\prod_{i=1}^m |G_i|\leq 4^{m}$, which implies that $l(\mathcal{G}_k\acts M)=k$. Since $k$ can be chosen to be arbitrarily large, we obtain that $l(M)=\infty$. 
\end{proof}

\begin{rem}
	The solvmanifold $M$ of \cref{solvmanifold inifite lenght} satisfies $\D(M)=0$. Indeed, since $M$ is aspherical and $\Gamma$ is polycyclic, we have $\D(M)\leq\rank\Zc\Gamma$. Since $\Gamma$ is isomorphic to the semi-direct product $\Z^2\rtimes_\phi\Z$, we can use \cite{daura2025IGG1} to conclude that $\Zc\Gamma$ is trivial. Consequently,  $\D(M)\leq\rank\Zc\Gamma=0$ and therefore $\D(M)=0$. 
\end{rem}

We study now the iterated length of locally symmetric spaces.

\begin{lem}\label{lenght locally symmetric spaces}
	Let $K\setminus G/\Gamma$ be a locally symmetric space where $G$ is a connected semisimple Lie group without compact factors, $K$ is a maximal compact subgroup and $\Gamma$ is a lattice. Then $l(K\setminus G/\Gamma)$ is bounded by a constant $C$ depending on $\Gamma$. 
\end{lem}

\begin{proof}
	Recall that if $\mu$ is the Haar measure of $G$ then $\vol(G/\Gamma)=\mu(F)$, where $F$ is a fundamental domain of $\Gamma$ in $G$. By \cite[11.9 Corollary]{raghunathan2012discrete}, there exists a constant $A$ such that $\vol(G/\Gamma)>A$ for all lattices of $G$. Moreover, if $\Gamma'$ is another lattice containing $\Gamma$ as a finite index subgroup then $\vol(G/\Gamma)=[\Gamma':\Gamma]\vol(G/\Gamma')$.
	
	Let $\mathcal{G}\acts K\setminus G/\Gamma$ be a free iterated action with $l(\mathcal{G})=l(\mathcal{G}\acts K\setminus G/\Gamma)=l$. Then $(K\setminus G/\Gamma)/\mathcal{G}\cong K\setminus G/\Gamma'$ where $\Gamma'$ is a lattice of $G$. Then $$\vol(G/\Gamma)=[\Gamma':\Gamma]\vol(G/\Gamma')=\prod_{i=1}^l|G_i|\vol(G/\Gamma')\geq 2^n A.$$
	
	In consequence $l\leq \log_2(\tfrac{\vol(G/\Gamma)}{A})$. The proof is finished by taking $C=\log_2(\tfrac{\vol(G/\Gamma)}{A})$.
\end{proof}

The dependence on the lattice cannot be removed. In order to give an example we need the following result:

\begin{prop}\cite[Proposition 2.3]{baker2001towers}\label{tower hyperbolic coverings}
	Let $N_0\longrightarrow N_1\longrightarrow \cdots \longrightarrow N_s$ be a tower of coverings of closed 3-manifolds. There exists a tower of coverings $M_0\longrightarrow M_1\longrightarrow \cdots \longrightarrow M_s$ of closed hyperbolic 3-manifolds and maps $f_i:M_i\longrightarrow N_i$  such that $\deg(f_i)=1$ and the covering $M_{i-1}\longrightarrow M_i$ is the pullback of $N_{i-1}\longrightarrow N_i$ by $f_i$ for all $i$.
\end{prop}

\begin{cor}\label{iterated action lenght degree 1 maps}
	Assume that we are in the setting of \cref{tower hyperbolic coverings}, $\pi_1(M_i)\trianglelefteq\pi_1(M_j)$ if and only if $\pi_1(N_i)\trianglelefteq\pi_1(N_j)$ for all $i<j$.
\end{cor}

Now, we take $N_i=M$ the solvmanifold from \cref{solvmanifold inifite lenght} for all $i$ and the covering $N_{i-1}\longrightarrow N_i$ the self-covering of \cref{solvmanifold inifite lenght}. By \cref{tower hyperbolic coverings}, for each action $\mathcal{G}_k=\{\Z/2\oplus\Z/2,\cdots \Z/2\oplus \Z/2\}\acts M$, there exists a closed hyperbolic $3$-manifold $M_{0,k}$ (which depends on $k$) such that we have a free iterated action $\mathcal{G}_k\acts M_{0,k}$. By \cref{iterated action lenght degree 1 maps}, we have $l(\mathcal{G}_k\acts M_{0,k})=k$. Since $M_{0,k}$ are hyperbolic manifolds for all $k$, the bound of \cref{lenght locally symmetric spaces} needs to depend on the lattice and not only the Lie group (in this example $\SO^+(3,1)$).

\begin{ques}
	Let $ K\setminus G/\Gamma$ be a locally symmetric space where $G$ is connected semisimple without compact factors and $\rank_\R G\geq 2$, $K$ is a maximal compact subgroup and $\Gamma$ is an irreducible lattice. Does there exist a constant $C$ only depending on $G$ such that $l(K\setminus G/\Gamma)\leq C$?
\end{ques}

The solvmanifold of \cref{solvmanifold inifite lenght} can be used to construct other closed aspherical locally homogeneous space $ K\setminus G/\Gamma$ such that $l( K\setminus G/\Gamma)=\infty$. 

\begin{prop}\label{lenght infty locally homogeneous space}
	There exists a closed aspherical locally homogeneous space $ K\setminus G/\Gamma$ such that the solvable radical of $G$ is abelian and $l( K\setminus G/\Gamma)=\infty$.
\end{prop} 

\begin{proof}
	Let $\Lambda$ be the fundamental group of a closed hyperbolic manifold of dimension $n\geq 3$ such that there exists an epimorphism $f:\Lambda\longrightarrow \Z$ and let $\phi:\Z\longrightarrow \Gl(2,\Z)$ be as in the proof of \cref{solvmanifold inifite lenght}. Then we can define $\Gamma_k=2^k(\Z^2)\rtimes_{\phi\circ f} \Lambda$. Using the Seifert construction (see \cite[Theorem 11.7.29]{lee2010seifert}) we can construct a closed aspherical locally homogeneous space with fundamental group $\Gamma_k$ for all $k$, which we denote by $M_k$. 
	
	The same arguments as in \cref{solvmanifold inifite lenght} show that all $\Gamma_k$ are isomorphic and since the Borel conjecture is true for lattices in connected Lie groups (see \cite{bartels2012borel,kammeyer2016farrell}) we can conclude that $M_k\cong M$ for all $k$. The inclusion $\Gamma_k\longrightarrow \Gamma_{k-1}$ induces a regular self-covering $M_{k}\longrightarrow M_{k-1}$. Finally, the tower of self-coverings $M_{k}\longrightarrow M_{k-1}\longrightarrow\cdots\longrightarrow M_0$ induces a free iterated group action $\mathcal{G}_k=\{\Z/2\oplus\Z/2,\cdots, \Z/2\oplus \Z/2\}\acts M$ such that $l(\mathcal{G}_k\acts M)=k$. Thus $l(M)=\infty$.
\end{proof}

\begin{ques}
	Let $K\setminus G/\Gamma$ be a closed aspherical locally homogeneous space where the solvable radical $R$ of $G$ is nilpotent and $G/R$ is semisimple without compact factors and $\rank_\R G/R\geq 2$. Does there exist a constant $C$ such that $l(K\setminus G/\Gamma)\leq C$?
\end{ques}


\end{document}